\documentclass{amsart}
\usepackage{amsmath, amssymb}

\input xy
\xyoption{all}

\newtheorem{theorem}{Theorem}[section]
\newtheorem{lemma}[theorem]{Lemma}
\newtheorem{corollary}[theorem]{Corollary}
\newtheorem{fact}[theorem]{Fact}
\newtheorem{proposition}[theorem]{Proposition}

\newtheorem*{thmA}{Theorem A}
\newtheorem*{thmAtypes}{Theorem A, reformulated}
\newtheorem*{thmB}{Theorem B}

\theoremstyle{definition}
\newtheorem{example}[theorem]{Example}
\newtheorem{remark}[theorem]{Remark}

\newtheorem*{definitionintro}{Definition}
\newtheorem{question}{Question}

\def\dcl{\operatorname{dcl}}
\def\acl{\operatorname{acl}}
\def\alg{\operatorname{alg}}
\def\tp{\operatorname{tp}}

\def\id{\operatorname{id}}

\def\Aut{\operatorname{Aut}}
\def\Fix{\operatorname{Fix}}

\def\U{\mathcal{U}}
\def\C{\mathcal{C}}
\def\K{\mathcal{K}}
\def\logd{\operatorname{log}_\delta}
\def\Ga{\mathbb G_{\operatorname{a}}}
\def\Gm{\mathbb G_{\operatorname{m}}}
 
\def\dx{\frac{d}{dx}}

\def\Ind#1#2{#1\setbox0=\hbox{$#1x$}\kern\wd0\hbox to 0pt{\hss$#1\mid$\hss}
\lower.9\ht0\hbox to 0pt{\hss$#1\smile$\hss}\kern\wd0}
\def\ind{\mathop{\mathpalette\Ind{}}}
\def\Notind#1#2{#1\setbox0=\hbox{$#1x$}\kern\wd0\hbox to 0pt{\mathchardef
\nn=12854\hss$#1\nn$\kern1.4\wd0\hss}\hbox to
0pt{\hss$#1\mid$\hss}\lower.9\ht0 \hbox to
0pt{\hss$#1\smile$\hss}\kern\wd0}
\def\nind{\mathop{\mathpalette\Notind{}}}

\begin{document}

\title[Logarithmic-differential pullbacks]{Internality of\\ logarithmic-differential pullbacks}

\author{Ruizhang Jin}
\address{Ruizhang Jin\\
University of Waterloo\\
Department of Pure Mathematics\\
200 University Avenue West\\
Waterloo, Ontario \  N2L 3G1\\
Canada}
\email{r6jin@uwaterloo.ca}

\author{Rahim Moosa}
\address{Rahim Moosa\\
University of Waterloo\\
Department of Pure Mathematics\\
200 University Avenue West\\
Waterloo, Ontario \  N2L 3G1\\
Canada}
\email{rmoosa@uwaterloo.ca}

\date{\today}

\thanks{{\em Acknowledgements}: R. Moosa was partially supported by NSERC Discovery  and DAS grants.}

\begin{abstract}
A criterion in the spirit of Rosenlicht is given, on the rational function $f(x)$, for when the planar vector field
$
\begin{Bmatrix}
y'=xy\\
x'=f(x)
\end{Bmatrix}
$
admits a pair of algebraically independent first integrals over some extension of the base field.
This proceeds from model-theoretic considerations by working in the theory of differentially closed fields of characteristic zero and asking: If $D\subseteq\mathbb A^1$ is a strongly minimal set internal to the constants, when is $\logd^{-1}(D)$, the pullback of $D$ under the logarithmic derivative, itself internal to the constants?
\end{abstract}

\maketitle

\setcounter{tocdepth}{1}
\tableofcontents

\section{Introduction}

\noindent
This paper is concerned with systems of differential equations of the form
\begin{equation}
\label{logd}
\begin{Bmatrix}
y'=xy\\
x'=f(x)
\end{Bmatrix}
\end{equation}
where $f$ is a rational function over a base differential field\footnote{All fields in this paper are of characteristic zero.} $(F,\delta)$.
Note that when the base is a field of constants, namely $\delta$ is trivial on $F$, these equations define a rational vector field on the plane.
Indeed, we will be applying our results to precisely that context.
But we are interested more generally in the case of a possibly nonconstant base field.
Note, also, that at least outside of $y=0$, the system~(\ref{logd}) can be expressed as the single second-order differential equation
$\displaystyle \left(\frac{y'}{y}\right)'=f\left(\frac{y'}{y}\right)$.
In other words, we are looking at the pullback of $x'=f(x)$ under the {\em logarithmic derivative} operator $y\mapsto \displaystyle\frac{y'}{y}$.

Other classes of second-order algebraic differential equations that have been the object of model-theoretic study in recent years include Painlev\'e equations in~\cite{np14, np17}, generic planar vector fields in~\cite{jaoui19}, and the twisted $D$-groups of~\cite[Example~3.4]{blsm18}.
A central feature of Painlev\'e equations and generic planar vector fields is that there is little structure induced on the solution space, whereas the system~(\ref{logd}) that we study here, like that of twisted $D$-groups, admits a lot of structure.\footnote{This will become clearer as we go along, but for the model theorist, what we have in mind here is disintegratedness versus nonorthogonality to the constants.}

What we ask about logarithmic-differential pullbacks is motivated by both model theory and differential-algebraic geometry:
\begin{question}
\label{question}
When is~(\ref{logd}) {\em almost internal} to the constants?
\end{question}
\noindent
We will review what this means formally in Section~\ref{sect-prelim}, but roughly speaking, to say that ~(\ref{logd}) is almost internal to the constants is to say that it admits two algebraically independent first integrals after  base change.
Recall that a {\em first integral} to a vector field $(\mathbb A^2,s)$ is a nonconstant rational function that is constant on the leaves of the corresponding foliation.
Equivalently, it is a nonconstant rational function on which the derivation induced by $s$ vanishes. 
So Question~1 asks: when does it happens that, over some differential field $(K,\delta)$ extending $(F,\delta)$, there exists a pair of algebraically independent rational functions in $K(x,y)$ that are constant for the (unique) derivation extending~$\delta$ on $K$ and taking $(x,y)\mapsto \big(f(x),xy\big)$.

Of course one should ask this question first about the single equation $x'=f(x)$.
When is $x'=f(x)$ almost internal to the constants?
When does it admit a first integral after base change?
In the case that $F$ is a field of constants the answer to this question is given by an old theorem of Rosenlicht (see~\cite[Proposition 2]{rosenlicht}): it is almost internal to the constants if and only if $f=0$ or
$\displaystyle \frac{1}{f}=\frac{d}{dx}(g)$, or
$\displaystyle \frac{1}{f}=\frac{c\frac{d}{dx}(g)}{g}$, for some $c\in F$ and $g\in F(x)$.
However, no extension of Rosenlicht's criterion to the case of nonconstant parameters is known -- the naive generalisations certainly fail.
So at this level of generality we will have to impose as an additional condition that $x'=f(x)$ is almost internal to the constants.
Or even that it is {\em internal} to the constants -- that after base change the rational function field is generated by a first integral.
Then we take the logarithmic-differential pullback of $x'=f(x)$, namely the system~(\ref{logd}), and ask when it too is almost internal to the constants.

An important example is $f=0$.
In that case~(\ref{logd}) defines a differential-algebraic group which extends the constant points of the multiplicative group by the constant points of the additive group.
It is not almost internal to the constants.
We do have one first integral, namely the rational function $x$ itself.
Moreover, each fibre $y'=cy$ has a first integral after we extend the base to a solution: if $\gamma$ is nonzero and satisfies $\gamma'=c\gamma$ then the rational function $\frac{y}{\gamma}$ is a first integral to $y'=cy$ over $F(\gamma)$.
(In model-theoretic parlance, this says that the system~(\ref{logd}) is {\em analysable} in the constants.)
But these first integrals on the fibres cannot be put together, and we do not obtain an algebraically independent pair of first integrals over any differential field.
An argument for this well known fact can be found in~\cite[Fact~4.2]{chtm}.
From the $f=0$ case one can deduce the same result for $f=\gamma$ where $\gamma\in F$.

Consider next the example $f(x)=x$.
Then, again, the logarithmic-differential pullback~(\ref{logd}) is not almost internal to the constants.
This is~\cite[Proposition~3.2]{jin2018}.
There is a first integral (namely, $\frac{x}{\gamma}$ where $\gamma'=\gamma$ is nonzero) but not a pair of algebraically independent first integrals.

Finally, consider $f(x)=x^2$.
Then $\frac{1}{x}+t$ is a first integral if $t'=1$.
On the other hand, $\frac{x}{y}$ is also a first integral to~(\ref{logd}).
Hence, in this case, the logarithmic-differential pullback {\em is} almost internal to the constants.
Note that $x'=x^2$ is differential-birationally isomorphic to $x'=1$ by the change of variable $x\mapsto\frac{-1}{x}$, but the latter has logarithmic-differential pullback {\em not} internal to the constants.
So, while almost internality to the constants is a differential-birational invariant of~(\ref{logd}), it is not an invariant of $x'=f(x)$.
The answer to Question~1 will have to be sensitive to the specific rational function $f$.

Our goal is to explain the different behaviour exhibited in the above examples.
Observe that when $f(x)=x^2$, the system~(\ref{logd}) is differential-birationally isomorphic to a product of first-order equations; namely, the change of variables $y\mapsto\frac{x}{y}$ yields
$\begin{Bmatrix}
y'=0\\
x'=x^2
\end{Bmatrix}.$
We say that the system~(\ref{logd}) {\em splits}.
A precise definition appears in the next section, but roughly speaking, it means that after a finite-to-finite differential-rational correspondence the system transforms into one of the form
$\begin{Bmatrix}
y'=\gamma y\\
x'=f(x)
\end{Bmatrix}$
for some $\gamma\in F$.
It is not hard to verify that the other examples we have considered (when $f$ is either constant or $x$) do not split.
It was at first expected that this was typical; that if a logarithmic-differential pullback is almost internal to the constants then it is for the concrete reason that it splits, see~\cite[Conjecture~5.4]{jinthesis}.
This turns out not to be entirely correct.

In fact, Question~\ref{question} depends sharply on a subtle model-theoretic invariant of the equation $x'=f(x)$ called its {\em binding group}.
This is a certain algebraic group, acting differential-rationally on the generic solutions to the equation, that generalises the differential Galois group in the case of linear ODE's.
Because the equation is first-order, it follows from general results that the dimension of the binding group is at most $3$.
We answer Question~1 in all but the top-dimensional case:

\begin{thmA}
Suppose $F$ is a differential field and $f\in F(x)$ is a rational function such that $x'=f(x)$ is internal to the constants with binding group not of dimension~$3$.
Then the logarithmic-differential pullback of $x'=f(x)$ is almost internal to the constants if and only if it splits.
\end{thmA}

It is necessary that the binding group not have dimension $3$ in the following strong sense: If the binding group {\em is} of dimension three -- and this does happen as we verify in Example~\ref{psl2} below -- then $x'=f(x)$ is a finite cover of another first-order algebraic differential equation whose logarithmic-differential pullback is almost internal to the constants but does {\em not} split.
This ubiquity of counterexamples in dimension $3$ appears as Theorem~\ref{3dim} below.

Our proof of Theorem~A goes through a careful case analysis by dimension of the binding group, studying dimensions $0,1$, and $2$ in turn.
This is done in Section~\ref{sect-not3dim}, where the bulk of the work of this paper is done.
(There are examples showing each case to be non-vacuous, see in particular Example~\ref{exdim2} in dimension~$2$.)

When the base is a field of constants it is known that the binding group is at most $1$-dimensional, and so Theorem~A applies to answer Question~\ref{question} entirely.
In this case we give explicit Rosenlicht-type criteria on $f(x)$ for~(\ref{logd}) to be almost internal to the constants:

\begin{thmB}
Suppose $F$ is an algebraically closed field of constants and $f\in F(x)$.
The rational vector field defined by
$\begin{Bmatrix}
y'=xy\\
x'=f(x)
\end{Bmatrix}$
is almost internal to the constants if and only if the following conditions on $f$ are satisfied:
\begin{itemize}
\item[(i)]
$f\neq 0$, and
\item[(ii)]
$\displaystyle \frac{1}{f}=\frac{d}{dx}(g)$ or $\displaystyle \frac{1}{f}=\frac{c\frac{d}{dx}(g)}{g}$ for some $c\in F$ and $g\in F(x)$, and
\item[(iii)]
$\displaystyle \frac{kx-e}{f}=\frac{\frac{d}{dx}(h)}{h}$ for some nonzero $k\in\mathbb Z$, $e\in F$ and $h\in F(x)$.
\end{itemize}
\end{thmB}

Condition~(ii) is the Rosenlicht condition on $f$ that we saw before.
Condition~(iii) can be expressed as saying that the partial-fraction decomposition of $\frac{x-\frac{e}{k}}{f}$ must be of the form $\sum_{i=1}^\ell\frac{r_i}{x-c_i}$ for some rational numbers $r_i$ and constants $c_i\in F$.
In particular, writing $f=\frac{P}{Q}$ with $P,Q$ coprime polynomials, condition~(iii) forces $\deg P\geq\deg Q+~2$.
(It also forces $P$ to have at most one multiple root, namely $\frac{e}{k}$ with multiplicity~$2$.)
In particular, we see that while all the examples $f(x)=c,x,x^2$ considered above satisfy~(i) and~(ii), only $f(x)=x^2$ satisfies condition~(iii).
This explains why the logarithmic-differential pullback of $x'=x^2$ behaves differently from the others.

Let us end this introduction by speculating on a more general setting in which this work could be situated.
The logarithmic derivative is a special case of the {\em generalised Schwarzian derivative} in the sense of~\cite{scanlon18}, arising from an algebraic group acting on a curve.
In this case it is the action of the multiplicative group on itself.
Now, the only faithful transitive algebraic group actions on curves are the one-dimensional groups acting on themselves (of which $\Ga$ gives rise to the derivative and $\Gm$ gives rise to the logarithmic deriviative), $\Gm\ltimes\Ga$ acting on the affine line by affine linear transformations, and $\operatorname{PSL}_2$ acting on the projective line by projective transformations (which gives rise to the classical Schwarzian derivative).
In all cases, by construction, the fibres of the generalised Schwarzian derivative are internal to the constants.
It would be natural to ask, in each of these cases, when is the generalised-Schwarzian-differential pullback of $x'=f(x)$ almost internal to the constants?

\medskip
\noindent
{\em Acknowledgements}.
We are grateful to R\'emi Jaoui and Anand Pillay for several useful conversations on this topic.

\bigskip
\section{Splitting and a formulation for types}
\label{sect-splitting}

\noindent
We work throughout this paper in a fixed sufficiently saturated model $(\U,\delta)$ of the theory of differentially closed fields in characteristic zero.
This is a large differentially closed field that serves as a universal domain for differential-algebraic geometry.
All parameter sets are assumed to be of cardinality strictly less than $|\U|$ unless explicitly stated otherwise.
All algebraic and differential-algebraic varieties are identified with their $\U$-points.
{\em Definable} will always mean definable with parameters in $(\U,\delta)$.
The field of constants is denoted by $\C=\{x\in\U:\delta x=0\}$.

By the {\em logarithmic derivative} we then mean the definable function $\displaystyle \logd x=\frac{\delta x}{x}$.
We have the following short exact sequence of differential-algebraic group homomorphisms
$$\xymatrix{
0\ar[r]&\Gm(\C)\ar[r]&\Gm\ar[r]^{\logd}&\Ga\ar[r]&0}
$$
where $\Gm$ denotes the multiplicative group and $\Ga$ the additive group.

For the sake of both clarity and rigour, it is more convenient to work with logarithmic-differential pullbacks of types rather than equations.

\begin{definitionintro}[Logarithmic-differential pullback]
For $p\in S_1(A)$ a complete $1$-type over $A$, we denote by $\logd^{-1}(p)$ the type $\tp(u/A)\in S_1(A)$ where $u\in\Gm$ is such that $\logd u\models p$ and $u\notin\acl(A,\logd u)$.
\end{definitionintro}

That $\logd^{-1}(p)$ exists is because for fixed $a\models p$ the equation $\delta y=ay$ has a solution outside of $\acl(A,a)$ by the existential closedness and saturation of $(\U,\delta)$.
Moreover, $\logd^{-1}(p)$ does not depend on the choice of $a$.
Indeed, suppose $a'\models p$ and $u'\notin\acl(A,a')$ is such that $\logd u'=a'$.
Then there is $\sigma\in\Aut_A(\U)$ with $\sigma(a)=a'$.
Now $\logd\big(\sigma(u)\big)=\sigma(\logd u)=\sigma(a)=a'$ and $\sigma(u)\notin\acl(A,a')$.
So $u'$ and $\sigma(u)$ are both generic realisations of the strongly minimal formula $\delta y=a'y$ over $\acl(A,a')$, and hence there is $\tau\in\Aut_{Aa'}(\U)$ with $\tau(\sigma(u))=u'$.
In particular, $\tp(u/A)=\tp(u'/A)$.

We now make precise what we mean by {\em splitting} in Theorem~A.

\begin{definitionintro}[Splitting]
\label{splits}
For $p\in S_1(A)$ we say that the logarithmic-differential pullback {\em splits} if we can factor some nontrivial integer power of $u\models\logd^{-1} p$ in $\Gm$ as $u^k=w_1w_2$ where $w_1\in\dcl(A,\logd u)$ and $\logd w_2\in\dcl(A)$.
\end{definitionintro}

Assuming $p$ itself is almost $\C$-internal, we have that $\tp(w_1/A)$ is almost $\C$-internal as any realisation is in the definable closure of a realisation of $p$.
Also $\tp(w_2/A)$ is $\C$-internal as any two realisations, having the same logarithmic derivative, will have a constant ratio.
Hence, if $\logd^{-1}(p)$ splits then any realisation is in the algebraic closure of realisations of almost $\C$-internal types.
That is, if $\logd^{-1}(p)$ splits then it is almost $\C$-internal.
The question is whether splitting is the only way $\logd^{-1}(p)$ can be almost $\C$-internal.

We restate Theorem~A with these formalities in place.

\begin{thmAtypes}
Suppose $F$ is an algebraically closed differential field and $f\in F(x)$ is such that the generic type $p\in S_1(F)$ of $\delta x=f(x)$ is $\C$-internal.
Suppose the binding group of $p$ is not of dimension $3$.
Then $\logd^{-1}(p)$ is almost $\C$-internal if and only if it splits.
\end{thmAtypes}

It is this version that we prove in Section~\ref{sect-main} below.

Finally, let us remark that it might be natural to ask about arbitrary minimal $\C$-internal $1$-types  $p$ -- without insisting that it is the generic type of an equation of the form $\delta x=f(x)$.
In fact, the only way in which we use this equation in the proof of Theorem~A is to rule out (by genus considerations) the possibility of the binding group of $p$ being an elliptic curve.
Our methods do not seem to yield a proof, nor a counterexample, when the binding group is an elliptic curve.

\bigskip
\section{Model-theoretic preliminaries}
\label{sect-prelim}

\noindent
Our aim in this section is to briefly recall various relevant notions from stability theory around internality and orthogonality, and then to establish some either elementary or well-known facts about them that will be used later.
We stay in the setting of differentially closed fields, though everything discussed here makes sense in, and is true of, stable first-order theories in general.
We suggest~\cite{marker91} for an introduction to the model theory of differentially closed fields.

By a {\em minimal} type we mean a stationary complete type of $U$-rank one.

Suppose $p\in S(A)$ is a stationary type.
We say that $p$ is {\em $\C$-internal} if for some $B\supseteq A$ and $a\models p$ independent of $B$ over $A$, $a\in\dcl(B,\C)$.
Often one uses the following ``fundamental system of solutions" characterisation: for some $n<\omega$ there is a Morley sequence $(a_1,\dots,a_n)\models p^{(n)}$ and an $(A,a_1,\dots,a_n)$-definable function $g(x)$ such that for all $a\models p$ there is a tuple $c$ from $\C$ such that $a=g(c)$.
Here $p^{(n)}$ is the (unique) type of an $n$-tuple of independent realisations of $p$.
A key tool in studying $\C$-internal types is the {\em binding group} which we will denote by $\Aut_A(p/\C)$.
By definition this is the group of permutations of $p(\U)$ that are induced by automorphisms of $\U$ fixing $A\cup\C$ pointwise.
When $p$ is $\C$-internal, the binding group, along with its action on $p(\U)$, is $A$-definable.
In fact, over possibly additional parameters, $\Aut_A(p/\C)$ is definably isomorphic to the $\C$-points of an algebraic group over the constants.

\begin{lemma}
\label{dclliason}
Suppose $p\in S(A)$ is a stationary $\C$-internal type.
If $a\models p$ and $b\in \dcl(Aa)$ then $q:=\tp(b/A)$ is a stationary $\C$-internal type and there is a surjective $A$-definable homomorphism $\phi:\Aut_A(p/\C)\to\Aut_A(q/\C)$.
If $p$ is minimal and $b\notin\acl(A)$ then $\ker(\phi)$ is finite.
\end{lemma}

\begin{proof}
Given $\sigma\in\Aut_A(p/\C)$ we define $\phi(\sigma)$ to be the restriction to $q^{\U}$ of any extension of $\sigma$ to $\Aut_A(\U/\C)$.
The fact that $b\in\dcl(Aa)$ ensures that $\phi(\sigma)$ does not depend on the extension chosen and is in fact $A$-definable.
Indeed, writing $b=f(a)$ for some $A$-definable function $f$, for any $b'\models q$ we have that $b'=f(a')$ for some $a'\models p$, and $\phi(\sigma)(b')=f(\sigma(a'))$ for any such $a'$.
It is also clear that $\phi$ is surjective: given $\tau\in\Aut_A(q/\C)$ extend it to $\hat\tau\in\Aut_A(\U/\C)$ and note that $\phi(\hat\tau|_{p(\U)})=\tau$.

Let $(a_1,\dots,a_n)\models p^{(n)}$ be a fundamental system of solutions for the $\C$-internality of $p$.
Then the action of an element of $\Aut_A(p/\C)$ on $\{a_1,\dots,a_n\}$ determines its action on $p(\U)$.
Now let $b_1,\dots,b_n\models q$ be such that $\tp(ab/A)=\tp(a_ib_i/A)$ for all $i=1,\dots, n$.
If $p$ is minimal and $b\notin\acl(A)$ then each $a_i\in\acl(Ab_i)$.
If $\sigma\in\ker(\phi)$ and $\hat\sigma\in\Aut_A(\U/\C)$ extends $\sigma$, then $\hat\sigma(b_i)=b_i$ for all $i$, and so there are only finitely many possible values that $\sigma$ can take on $\{a_1,\dots,a_n\}$.
This proves that $\ker(\phi)$ is finite.
\end{proof}

\begin{lemma}
\label{quotientliason}
Suppose $p\in S(A)$ is a stationary type that is $\C$-internal and whose binding group $G$ acts regularly -- i.e., uniquely transitively -- on $p(\U)$.
Suppose $L$ is an $A$-definable subgroup of~$G$.
Then $L$ is normal and there exist $a\models p$ and $e\in\dcl(A,a)$ such that the binding group of $\tp(e/A)$ is $G/L$.
\end{lemma}

\begin{proof}
Let $a\models p$ and set $e$ to be a canonical parameter for $L\cdot a$, the orbit of $a$ under the action of $L$.
Then clearly $e\in\dcl(A,a)$, and hence $q:=\tp(e/A)$ is stationary and $\C$-internal.
Let $H:=\Aut_A(q/\C)$ be its binding group and let $\phi:G\to H$ be the surjective definable homomorphism given by Lemma~\ref{dclliason}.
We show that $L=\ker(\phi)$.
Fix $g\in G$ and $\sigma\in\Aut_A(\U/\C)$ extending~$g$.
Then
\begin{eqnarray*}
g\in\ker(\phi)
& \implies &
\sigma(e)=e\\
& \implies &
\sigma(L\cdot a)=L\cdot a\\
&\implies &
L\cdot(ga)=L\cdot a\\
&\implies&
ga=\ell a\ \ \ \ \ \text{ for some }\ell\in L\\
&\implies&
g\in L\ \ \ \text{ by the regularity of the action.}
\end{eqnarray*}
Conversely, suppose $g\in L$.
Let $e'\models q$ be arbitrary.
Then $e'$ is the canonical parameter for $L\cdot a'$ for some $a'\models p$.
Now
$g(L\cdot a')=L\cdot a'$ as $g\in L$.
Hence $\sigma(L\cdot a')=L\cdot a'$, and so $\sigma(e')=e'$.
That is, $\phi(g)=1$, as desired.
\end{proof}

\begin{lemma}
\label{normal}
Suppose $p\in S(A)$ is a stationary type that is $\C$-internal with binding group $G$.
Suppose $H$ is an algebraic group over the constants and $\phi:H(\C)\to G$ is a definable isomorphism over possibly additional parameters.
If $L$ is a normal algebraic subgroup of $H$ over the constants then $\phi\big(L(\C)\big)$ is definable over $A\cup\C$.
\end{lemma}

\begin{proof}
For convenience, let us take $A=\emptyset$.
Recall that the binding group appears as the group of automorphisms of an object in the {\em binding groupoid} as introduced by Hrushovski in~\cite{hrushovskigroupoid} -- but see also~\cite[$\S$3]{groupoid} for a detailed exposition with present terminology.
Namely, we have a $0$-definable connected groupoid $\mathcal G$ with objects $(O_i)_{i\in I}$ and morphisms $(f_m:O_i\to O_j)_{m\in M(i,j)}$ all in the pure field of constants, plus one new object $O_*$  with morphisms $(f_m:O_i\to O_*)_{m\in M(i,*)}$ and $(f_m:O_*\to O_j)_{m\in M(*,j)}$ in $\U$, such that $G=\Aut_{\mathcal G}(O_*)$, $H(\C)=\Aut_{\mathcal G}(O_c)$, and $\phi=f_m$ for some fixed $c\in I$ and $m\in M(c,*)$.
Now, each morphism $f_{m'}:O_c\to O_*$ induces a definable isomorphism $F_{m'}:\Aut_{\mathcal G}(O_c)\to \Aut_{\mathcal G}(O_*)$ given by $\alpha\mapsto f_{m'}\alpha f_{m'}^{-1}$.
Since $L(\C)$ is normal in $\Aut_{\mathcal G}(O_c)$ we have that $F_{m}\big(L(\C)\big)=F_{m'}\big(L(\C)\big)$ for all $m'\in M(c,*)$.
In particular, or all $\sigma\in\Aut(\U/\C)$,
\begin{eqnarray*}
\sigma\phi\big(L(\C)\big)&=&\sigma F_{m}\big(L(\C)\big)\\
&=&F_{\sigma(m)}\big(L(\C)\big)\\
&=&F_{m}\big(L(\C)\big)\\
&=&\phi\big(L(\C)\big)
\end{eqnarray*}
as desired.
\end{proof}

Suppose $p\in S(A)$ is a stationary type.
We say that $p$ is {\em weakly orthogonal to~$\C$} if $a\ind_A\C$ for all $a\models p$.
An equivalent characterisation is that of ``having no new constants"; namely, that $\dcl(A,a)\cap\C=\dcl(A)\cap\C$.
Note that by the fundamental system of solutions characterisation of internality, if $p$ is $\C$-internal then $p^{(n)}$ is not weakly $\C$-orthogonal for some $n<\omega$.

Because of stationarity, if $p$ is both $\C$-internal and weakly $\C$-orthogonal, then $\Aut_A(p/\C)$ acts transitively on $p(\U)$.
In particular, 
being the orbit of an element under the definable action of a definable group, $p(\U)$ is a definable set and so $p$ is isolated.
If $p$ is in addition minimal then $p(\U)$ will be a strongly minimal set.

\begin{lemma}
\label{dimliason}
Suppose $p\in S(A)$ is minimal $\C$-internal.
If $p^{(n)}$ is weakly orthogonal to $\C$ but $p^{(n+1)}$ is not, then $\dim\Aut_A(p/\C)=n$.
\end{lemma}

\begin{proof}
Note, first of all, that $G=\Aut_A(p^{(m)}/\C)$ for all $m\geq 1$.
So if $p^{(n)}$ is weakly orthogonal to $\C$ then the action of $G$ on $p^{(n)}(\U)$ is transitive, and hence $\dim G\geq n$.
Now, let $(a_1,\dots,a_n)\models p^{(n)}$.
We can extend this to a Morley sequence $(a_1,\dots,a_m)$ that forms a fundamental system of solutions, with $m> n$.
In particular, $G$ is determined by its action on $(a_1,\dots,a_m)$.
If we let $q$ be the nonforking extension of $p$ to $B:=A\cup\{a_1,\dots,a_n\}$ and set $K:=\Fix(a_1,\dots,a_n)=\{g\in G:ga_i=a_i\}$, then $K$ embeds into $\Aut_B(q^{(m-n)}/\C)$.
But as $p$ is minimal and $p^{(n+1)}$ is not weakly orthogonal to $\C$ we have that $p(\U)\subseteq\acl(B,\C)$, and hence $\Aut_B(q^{(m-n)}/\C)$ is a finite group.
So $K$ is finite.
Since every element of $G$ is determined modulo~$K$, by its action on $(a_1,\dots,a_n)$, this means that $G\subseteq\acl(A,a_1,\dots,a_n)$.
Hence $\dim G=n$, as desired.
\end{proof}

\begin{corollary}
\label{dimliason-cor}
Suppose $p\in S(A)$ is minimal $\C$-internal.
Then
$$\dim\Aut_A(p/\C)=\max\{m<\omega: p^{(m)}\text{ is weakly orthogonal to }\C\}.$$
\end{corollary}

\begin{proof}
As we have pointed out, $\C$-internality implies that some power is not weakly $\C$-orthogonal.
Note also that if $p^{(m)}$ is not weakly $\C$-orthogonal then neither is $p^{(m+1)}$.
The Corollary now follows from Lemma~\ref{dimliason}.
\end{proof}

A more robust notion than internality is {\em almost internality}:
$p$ is almost $\C$-internal if for some $B\supseteq A$ and $a\models p$ independent of $B$ over $A$, $a\in\acl(B,\C)$.
So $\acl$ has replaced $\dcl$ in the definition.

\begin{lemma}
\label{iai}
  Every almost $\C$-internal type over~$A$ is interalgebraic over $A$ with a $\C$-internal type over $A$.\end{lemma}
\begin{proof}
This is well known, but as we could not find a detailed proof we provide one here.
Suppose $p\in S(A)$ is almost $\C$-internal and $a\vDash p$.
Let $n$ be least such that
  there exists an $L_A$-formula $\varphi(x,y,z)$,
  a tuple $b$ independent from $a$ over $A$ and
  a tuple $c$ from $\C$ such that
  $\vDash\varphi(a,b,c)$ and $|\varphi(\mathcal U,b,c)|=n$.
  We fix these $b$, $c$, and $\varphi$.

  We claim that $\varphi\left(\mathcal U,b,c\right)\subseteq\mathrm{acl}(Aa)$.
  Indeed, let $a=a_1,a_2,\ldots ,a_n$ be the elements of $\varphi(\mathcal U,b,c)$. Towards a contradiction,
  suppose without loss of generality
  that $a_2\not\in\mathrm{acl}(Aa)$. Then there are $a_2'$, $b'$ and $c'$ such that
  $\tp(a_2'b'c'/Aa)=\tp(a_2bc/Aa)$ and
  $a_2'b'\ind_{Aa} a_2\cdots a_nb$.
  We want to show that $a\models \varphi(x,b,c)\wedge\varphi(x,b',c')$ also witnesses the almost $\C$-internality of $p$.
  This would suffice as it would contradict the minimality of $n$.
  We need that $bb'\ind_A a$.
  Since $b\ind_A a$ and $b\ind_{Aa} b'$,
  we have $b\ind_{Ab'} a$.
 But also, $b'\ind_A a$.
 Hence, $bb'\ind_A a$, as desired.

 Now, let $d$ be a code for the finite set $\varphi\left(\U,b,c\right)$.
Then $a\in\acl(Ad)$ and as each $a_i\in \acl(Aa)$ we get that $d\in \acl(Aa)$ as well.
So $\acl(Aa)=\acl(Ad)$.
Moreover, it follows that $d\ind_A b$, and as $d\in\dcl(A,b,c)$ we have that $\tp(d/A)$ is $\C$-internal.
\end{proof}

\begin{lemma}
\label{aim}
Suppose $p\in S(A)$.
The following are equivalent:
\begin{itemize}
\item[(i)]
$p$ is almost $\C$-internal,
\item[(ii)]
for some $m>0$, $p^{(m)}$ is not weakly orthogonal to $\C$,
\item[(iii)]
for some $m>0$ and $(a_1,\dots,a_m)\models p^{(m)}$, $a_m\in\acl(A,\C,a_1,\dots,a_{m-1})$.
\end{itemize}
\end{lemma}

\begin{proof}
That~(iii) implies~(ii) with the same $m$ is clear.
For~(ii) implies~(i) note that by minimality, $(a_1,\dots,a_m)\nind_A\C$ yields $a_i\in\acl(A,\C,a_1,\dots,a_{i-1})$ for some $i\leq m$.

Finally, assume~(i) holds.
Let $p'\in S(A)$ be $\C$-internal and interalgebraic with~$p$ over $A$, by Lemma~\ref{iai}.
If $(a_1',\dots, a_n')$ is a fundamental system of solutions for $p'$ and $a'_{n+1}\models p'$ is independent of $(a_1',\dots,a_n')$, then $a_{n+1}'\in\dcl(A,\C,a_1',\dots, a_n')$.
Let $a_1\models p$ be interalgebraic with $a_1'$ over $A$, and extend it to $(a_1,\dots,a_{n+1})$ such that $\big((a_1,a_1'),\dots,(a_{n+1},a_{n+1}')\big)$ is a Morely sequence.
Then $a_{n+1}\in\acl(A,a_{n+1}')\subseteq\acl(A,\C,a_1',\dots, a_n')=\acl(A,\C,a_1,\dots, a_n)$, which is~(iii) with $m=n+1$.
\end{proof}

As suggested by Corollary~\ref{dimliason-cor} and Lemma~\ref{aim},
when working with almost $\C$-internality, rather than $\C$-internality, the role of the dimension of the binding group is played by $\max\{m<\omega: p^{(m)}\text{ is weakly orthogonal to }\C\}$.
This will agree with the dimension of the binding group of a $\C$-internal type interalgebraic with $p$.

The following is the manifestation in differentially closed fields of Hrushovski's~\cite{udi89} classification of faithful transitive group actions on strongly minimal sets in stable theories, along with the fact that every infinite finite rank definable field in a differentially closed field is definably isomorphic to the field of constants (see~\cite[Corollary~1.6]{pillaydag}).

\begin{fact}
\label{homspace}
Suppose $(G,S)$ is a definable faithful transitive group action of $G$ on a strongly minimal set $S$.
Then
\begin{itemize}
\item[(1)]
$\dim G=1$ and $G$ acts regularly, or
\item[(2)]
$\dim G=2$ and $(G,S)$ is definably isomorphic to the action of $\Gm(\C)\ltimes\Ga(\C)$ on $\C$ by affine transformations, or
\item[(3)]
$\dim G=3$ and $(G,S)$ is definably isomorphic to the action of $\operatorname{PSL}_2(\C)$ on $\mathbb P(\C)$ by projective transformations.
\end{itemize}
In particular, $\dim G\leq 3$.
\end{fact}

\begin{corollary}
\label{atmost3}
Suppose $p\in S(A)$ is a minimal type that is almost $\C$-internal.
Then $p^{(4)}$ is not weakly orthogonal to $\C$.
\end{corollary}

\begin{proof}
By Lemma~\ref{iai} we may assume that $p$ is $\C$-internal.
Let $G=\Aut_A(p/\C)$ be the binding group.
By Corollary~\ref{dimliason-cor}, it suffices to show that $\dim G\leq 3$.
But if $\dim G>0$ then $p$ is weakly orthogonal to $\C$ so that $G$ acts transitively on $S:=p(\U)$.
Now apply Fact~\ref{homspace}.
\end{proof}

\medskip

Finally, the following is a useful consequence of stable embededness.

\begin{lemma}
\label{inc}
Suppose $\phi(z)$ is a formula over $A$ in variables $z=(z_1,\dots,z_m)$.
If $\phi(z)$ has a realisation in $\C^m$ then it has a realisation in $\big(\acl(A)\cap\C\big)^m$.
\end{lemma}

\begin{proof}
By stable embeddeness, $\phi(\C)=\psi(\C)$ where $\psi(z)$ is an $L_{\operatorname{ring}}$-formula over $\dcl(A)\cap\C$.
As $K:=\acl(A)\cap\C$ is an algebraically closed field it is an elementary substructure of $\C$ and hence $\psi(K)$ is nonempty.
\end{proof}

\bigskip
\section{Counterexamples in dimension $3$}
\label{sect-3dim}

\noindent
We show in this section that it is possible for $\logd^{-1}(p)$ to be almost $\C$-internal without splitting.
Indeed, this happens near any minimal $\C$-internal $p$ for which $p^{(3)}$ is weakly $\C$-orthogonal.

First a preparatory lemma.

\begin{lemma}
\label{gmaction}
Suppose $F$ is a differential field with algebraically closed constants and $p\in S(F)$ is a $\C$-internal and weakly $\C$-orthogonal minimal type.
\begin{itemize}
\item[(a)]
The binding group of $p$ is definably isomorphic to $\Ga(\C)$ if and only if there are $a\models p$ and $b\in F\langle a\rangle \setminus F^{\alg}$ such that $\delta b\in F$.
\item[(b)]
The binding group of $p$ is definably isomorphic to $\Gm(\C)$ if and only if there are $a\models p$ and $b\in F\langle a\rangle \setminus F^{\alg}$ such that $\logd b\in F$.
\end{itemize}
\end{lemma}

\begin{proof}
Let $S:=p(\U)$.
We know by weak $\C$-orthogonality that $S$ is a strongly minimal set acted upon transitively by the binding group $G:=\Aut_F(p/\C)$.

First, suppose there are $a\models p$ and $b\in F\langle a\rangle \setminus F^{\alg}$ such that $\gamma:=\delta b\in F$.
Since $p$ is weakly $\C$-othogonal, so is $q(x):=\tp(b/F)$.
If $\beta\in F^{\alg}$ satisfied $\delta \beta=\gamma$, then $b-\beta$ would be a constant that is clearly dependent on $b$ over $F$.
Hence the formula $\delta x=\gamma$ has no realisations in $F^{\alg}$.
So this formula isolates $q(x)$ and $\Aut_F(q/\C)$ is definably isomorphic to $\Ga(\C)$.
But by Lemma~\ref{dclliason} there is a surjective definable homomorphism $\phi:G\to\Aut_F(q/\C)$ with finite kernel.
As the additive group has no proper finite covers, $\phi$ must be an isomorphism, and we have that $G$ is definably isomorphic to $\Ga(\C)$, as desired.

Similarly, if $\gamma:=\logd b\in F$, then $q$ will be isolated by $\logd x=\gamma$ and will have binding group $\Gm(\C)$.
We obtain a definable homomorphism $\phi:G\to\Gm(\C)$ that is surjective and with finite kernel.
But as the only finite covers of the multiplicative group are those of the form $[n]:\Gm\to\Gm$ that raise to the $n$th power, it follows that $G$ is definably isomorphic to $\Gm(\C)$ in this case.

For the converse directions of both parts, we first show how the $1$-dimensionality of $G$ implies the existence of infinitely many automorphisms of the differential field $F\langle a\rangle$ that fix $F$ pointwise.
Indeed, the action of $G$ on $S$ is uniquely transitive by Fact~\ref{homspace}.
So if we fix $a\in S$ then $S\subseteq\dcl(F,a,\C)$.
We thus have an $F\langle a\rangle$-definable embedding $f:S\to\C^n$ for some $n>0$.
It follows that the generic type of $S$ over $F\langle a\rangle$ is not isolated and so $S(F\langle a\rangle^{\alg})$ is infinite.
But
$f\big(S(F\langle a\rangle^{\alg})\big)\subseteq \dcl\big(F\langle a\rangle^{\alg}\cap\C\big)=\dcl(F\cap\C)$
by the weak $\C$-orthogonality of~$p$ and the algebraic closedness of $F\cap\C$.
It follows that $S(F\langle a\rangle^{\alg})=S(F\langle a\rangle)$, and so the latter is infinite.
Note that if $e\in S(F\langle a\rangle)\setminus F^{\alg}$ then because $p$ is minimal we must have $a\in S(F\langle e\rangle^{\alg})=S(F\langle e\rangle)$.
In particular, $F\langle a\rangle= F\langle e\rangle$, and $a\mapsto e$ induces an automorphism of the differential field $F\langle a\rangle$ that fixes $F$ pointwise.
We have shown there are infinitely many such automorphisms, as desired.

This allows us to apply a theorem
of Kolchin~\cite[Theorem~7]{kolchin53} (see also~\cite{matsuda}),
to conclude\footnote{Note that $F$ is relatively algebraically closed in $F\langle a\rangle$ as $p$ is stationary, and the constants of $F\langle a\rangle$ are the same as those of $F$ as $p$ is weakly $\C$-orthogonal -- so Kolchin's theorem does apply.}
 that there exists $b\in F\langle a\rangle\setminus F^{\alg}$ such that
\begin{itemize}
\item[(1)]
$\gamma:=\delta b\in F$, or
\item[(2)]
$\gamma:=\logd b\in F$, or
\item[(3)]
$(\delta b)^2=\lambda b(b^2-1)(b-\kappa)$ for some nonzero $\lambda,\kappa\in F$.
\end{itemize}
Letting $q:=\tp(b/F)$  we have that $q(x)$ is the generic type of the formula
\begin{itemize}
\item[(1)]
$\delta x=\gamma$, or
\item[(2)]
$(\delta x=\gamma x)\wedge(x\neq 0)$, or
\item[(3)]
$(\delta x)^2=\lambda x(x^2-1)(x-\kappa)$,
\end{itemize}
respectively.
Moreover, $q$ is weakly $\C$-orthogonal.
As we have seen, it then follows that if we are in case~(1) the binding group of $q$ is definably isomorphic to $\Ga(\C)$, and if we are in case~(2) it is definably isomorphic to $\Gm(\C)$.
It is also well known that  in case~(3) the binding group of $q$ cannot be either $\Ga(\C)$ or $\Gm(\C)$, see for example the proof of Proposition~5.28 in~\cite{jinthesis} for an explicit argument.
Now, by Lemma~\ref{dclliason} there is a definable homomorphism $\phi:G\to\Aut_F(q/\C)$ that is surjective and has finite kernel.
If $G$ is definably isomorphic to $\Ga(\C)$ then $\phi$ is an isomorphism and so $\Aut_F(q/\C)$ is also $\Ga(\C)$ and we are in case~(1).
Similarly, if $G$ is $\Gm(\C)$ then so is $\Aut_F(q/\C)$ and we are in case~(2).
\end{proof}

Next we point out that case~(3) of Fact~\ref{homspace}, where $\operatorname{PSL}_2(\C)$ acts on $\mathbb P(\C)$ by projective transformations, is realised as a binding group in differentially closed fields.
It is the binding group of the equation $\delta x=f(x)$ where $f$ is the differentially generic degree $2$ polynomial.
Indeed, this is probably well-known and the following argument follows suggestions to the first author by Anand Pillay.

\begin{example}
\label{psl2}
Fix independent differentially transcendental $\alpha_0,\alpha_1,\alpha_2$, and let $p(z)$ be the generic type of $\delta z=\alpha_2z^2+\alpha_1z+\alpha_0$ over $\mathbb Q\langle\alpha_1,\alpha_2,\alpha_3\rangle$.
Then $p$ is $\C$-internal with binding group isomorphic to $\operatorname{PSL}_2(\C)$.
\end{example}

\begin{proof}
Consider first the generic type $q(x,y)\in S_2(F)$ of the following system of linear differential equations:
$
\begin{Bmatrix}
x'=\beta_{1,1}x+\beta_{1,2}y\\
y'=\beta_{2,1}x+\beta_{2,2}y
\end{Bmatrix}
$
where the $\beta_{i,j}$ are independent differentially transcendental elements and $F:=\mathbb Q\langle \beta_{1,1}, \beta_{1,2}, \beta_{2,1}, \beta_{2,2}\rangle$.
Then $q$ is $\C$-internal by linearity.
We show that $\Aut_F(q/\C)$ is isomorphic to $\operatorname{GL}_2(\C)$.
Fix a pair of independent realisations $(a_1,b_1),(a_2,b_2)$, and form the matrix
$V=
\begin{pmatrix}
a_1 &a_2\\
b_1 &b_2
\end{pmatrix}.
$
For notational convenience, we write that $V\models q^{(2)}$ and that $\delta V=BV$ where
$B=(\beta_{ij})$.
By genericity of $q$ and independence of the ralisations, the entries of $V$ form an algebraically independent set of transcendental elements over $F$.
In particular, $V$ is invertible.
Given $\sigma\in \Aut_F(q/\C)$, we have $\delta V^\sigma=BV^\sigma$.
It follows that $M_\sigma:=V^{-1}V^\sigma\in\operatorname{GL}_2(\C)$, and that $\sigma\mapsto M_\sigma$ is a homomorphism from $\Aut_F(q/\C)$ to $\operatorname{GL}_2(\C)$ defined over $F(V):=F(a_1,b_1,a_2,b_2)$.
It is injective because the columns of $V$ form a basis for the solution set to the linear differential system over $\C$, and so if $V^\sigma=V$ then $\sigma=\id$.
So much is true for any $2\times 2$ system of linear differential equations.
For surjectivity we will use the differentially-generic choice of the entries of $B$.
Namley, since $B=(\delta V)V^{-1}$, the entries of $V$ are also independent differential transcendentals, and so $q^{(2)}$ is weakly $\C$-orthogonal.
In particular, if $M\in\operatorname{GL}_2(\C)$ is arbitrary then $V$ is independent of $M$ over $F$, and so the columns of $W:=VM$, which are clearly solutions to the linear differential system, are in fact  independent generic solutions.
That is, $W\models q^{(2)}$.
By weak $\C$-orthogonality again, $W=V^\sigma$ for some $\sigma\in \Aut_F(q/\C)$, and hence $M=M_\sigma$.
This completes the proof that $\Aut_F(q/\C)$ is isomorphic to $\operatorname{GL}_2(\C)$.

Now, consider $(a,b)\models q$, set $e:=\frac{a}{b}$, and let $p(z):=\tp(e/F)$.
A direct computation shows that
$$\delta e=-\beta_{2,1}e^2+(\beta_{1,1}-\beta_{2,2})e+\beta_{1,2}.$$
Setting $\alpha_0:=\beta_{1,2}$, $\alpha_1:=\beta_{1,1}-\beta_{2,2}$, and $\alpha_2:=-\beta_{2,1}$, we have that $p(z)$ is the generic type of $\delta z=\alpha_2z^2+\alpha_1z+\alpha_0$ over $F$, and $\{\alpha_0,\alpha_1,\alpha_2\}$ are independent differentially transcendental elements.
By Lemma~\ref{dclliason} we have an induced $F$-definable surjective homomorphism $\phi:\Aut_F(q/\C)\to\Aut_F(p/\C)$.

It is clear that under the identification of $\Aut_F(q/\C)$ with $\operatorname{GL}_2(\C)$ given above, $\Gm(\C)\subseteq\ker(\phi)$.
For the reverse containment, suppose $\sigma\in\ker(\phi)$.
Then
\begin{eqnarray*}
\delta\left(\frac{\sigma(a)}{a}\right)
&=&
\frac{a(\beta_{1,1}\sigma(a)+\beta_{1,2}\sigma(b))-\sigma(a)(\beta_{1,1}a+\beta_{1,2}b)}{a^2}\\
&=&
\frac{\beta_{1,2}(a\sigma(b)-\sigma(a)b)}{a^2}\\
&=&
0\ \ \ \ \ \ \ \ \ \ \ \ \text{ since $\sigma\left(\frac{a}{b}\right)=\frac{a}{b}$.}
\end{eqnarray*}
Hence $\sigma(a)=ca$ for some constant $c\in\C$.
Since $\sigma\left(\frac{a}{b}\right)=\frac{a}{b}$, we have $\sigma(b)=cb$ as well.
In principle, this constant $c$ may depend on the choice of $(a,b)$.
So let $c_1,c_2\in\mathcal C$ be such that $\sigma(a_1,b_1)=c_1(a_1,b_1)$ and $\sigma(a_2,b_2)=c_2(a_2,b_2)$, where $(a_1,b_1),(a_2,b_2)$ are the independent realisations of $q$ chosen before.
Then $(a_1+a_2,b_1+b_2)\models q$ as well, and we get  $\sigma(a_1+a_2,b_1+b_2)=d(a_1+a_2,b_1+b_2)$ for some $d\in\mathcal C$.
The $\C$-linear independence of $(a_1,b_1), (a_2,b_2)$ implies that $d=c_1=c_2$.
Hence,
$V^\sigma=
d
\begin{pmatrix}
a_1 &a_2\\
b_1 &b_2
\end{pmatrix}
$
and so
$\displaystyle M_\sigma
=
V^{-1}V^\sigma
=
\begin{pmatrix}
d &0\\
0 &d
\end{pmatrix}
\in\Gm(\C)$.
We have shown that $\ker(\phi)=\Gm(\C)$.
It follows $\Aut_F(p/\C)$ is isomorphic to $\operatorname{PSL}_2(\C)$, as desired.
\end{proof}

Fix a differential field $F$ and a minimal type $p\in S_1(F)$ that is $\C$-internal and has a $3$-dimensional binding group.
By the above example such an $F$ and $p$ exist.

\begin{proposition}
\label{inf1}
Suppose $(a_1,a_2,a_3)\models p^{(3)}$.
There exists an extension $F'$ of $F$ by constants and $v\in F'\langle a_1,a_2,a_3\rangle\setminus F'\langle a_1\rangle^{\alg}$ such that  $\logd v\in F'\langle a_1\rangle$.
\end{proposition}

\begin{proof}
Let $q$ be the nonforking extension of $p$ to $F\langle a_1\rangle$, and let $G=\Aut_{F\langle a_1\rangle}(q/\C)$ be the binding group of $q$.
Since $p^{(2)}$ is weakly orthogonal to~$\C$ we get that $q$ is weakly orthogonal to~$\C$, and hence $q^{\U}$ is a strongly minimal set that is acted upon transitively by $G$.
Since $p^{(3)}$ is weakly orthogonal to~$\C$ but $p^{(4)}$ is not (by~\ref{dimliason-cor}), we have that $q^{(2)}$ is weakly orthogonal to $\C$ but $q^{(3)}$ is not.
Lemma~\ref{dimliason} applied to~$q$ thus implies that $\dim G=2$.

Now, Fact~\ref{homspace} tells us that the action of $G$ on $q^{\U}$ is definably isomorphic to $\Gm(\C)\ltimes\Ga (\C)$ acting on $\C$ by affine transformations.
As this action is sharply $2$-transitive, we have that $G$ acts regularly on the set of realisations of $q^{(2)}$.
Moreover, we have a normal definable subgroup $L$ of $G$ such that $G/L$ is definably isomorphic to $\Gm(\C)$.
In principle the parameters from $L$ could come from anywhere.
But as $L$ is normal, Lemma~\ref{normal} applies to $q^{(2)}$ to give us that $L$ is defined over $F\langle a_1\rangle\cup\C$.
 
Let $F'$ be an extension of $F$ by constants such that $F'\cap\C$ is algebraically closed and $L$ is defined over $K:=F'\langle a_1\rangle$.
Let $r$ be the nonforking extension of $q^{(2)}$ to $K$.
Because $q^{(2)}$ is weakly orthogonal to $\C$, we have that $r=\tp(a_2,a_3/K)$ is implied by~$q^{(2)}$.
Hence $G=\Aut_{K}(r/\C)$ and it acts regularly on $r^{\U}$.
As $L$ is a $K$-definable subgroup of $G$, Lemma~\ref{quotientliason} applies to $r$, and we get $e\in K\langle a_2,a_3\rangle$ such that the binding group of $\tp(e/K)$ is $G/L=\Gm(\C)$.
Lemma~\ref{gmaction}(b) applies\footnote{Note that by weak $\C$-orthgonality of $p$ the constants of $K$ are those of $F'$ and hence algebraically closed.}
 to $\tp(e/K)$ and there must exist $v\in K\langle e\rangle \setminus K^{\alg}$ such that $\logd v\in K$.
As $K=F'\langle a_1\rangle^{\alg}$ and $K\langle e\rangle\subseteq F'\langle a_1,a_2,a_3\rangle$, this proves the proposition.
\end{proof}

We can now produce a $\C$-internal minimal $1$-type whose logarithmic-differential pullback is $\C$-internal but does not split.
Namely, with notation as in the proposition, set $b:=\logd v$ and consider $p':=\tp(b/F')$.
Then $b$ is in the definable closoure of $a_1$ over $F'$, and hence $p'$ is minimal and $\C$-internal.
The logarithmic-differential pullback $\logd^{-1}(p')=\tp(v/F')$ is also $\C$-internal as $v\in F'\langle  a_1,a_2,a_3\rangle$.
If $\logd^{-1}(p')$ splits then we can factor some nontrivial integer power of $v$ in $\Gm$ as $v^k=w_1w_2$ where $w_1\in F'\langle b\rangle$ and $\logd w_2\in F'$.
It follows that $w_2\in F'\langle v\rangle$, and so $r:=\tp(w_2/F')$ is $\C$-internal .
Note that if $w_2\in F'^{\alg}$, then as $w_1\in F'\langle b\rangle\subseteq F'\langle a_1\rangle$, we would have $v\in F'\langle a_1\rangle^{\alg}$.
As that is not the case, we must have $w_2\notin F'^{\alg}$.
So Lemma~\ref{gmaction}(b) implies that $\Aut_{F'}(r/\C)$ is definably isomorphic to $\Gm(\C)$.
But there is a surjective $F'$-definable homomorphism $\Aut_{F'}\big(\logd^{-1}(p')/\C\big)\to \Aut_{F'}(r/\C)$ by Lemma~\ref{dclliason}.
That is, $\Aut_{F'}\big(\logd^{-1}(p')/\C\big)$ has $\Gm(\C)$ as a definable quotient.
On the other hand, $v\in F'\langle  a_1,a_2,a_3\rangle$ implies that $\Aut_{F'}\big(\logd^{-1}(p')/\C\big)$ is in turn a definable quotient of $\Aut_{F'}\big(p^{(3)}|_{F'}/\C\big)=\Aut_{F}(p/\C)$, where the last equality uses the weak $\C$-orthogonality of $p^{(3)}$ and the fact that $F'$ is an extension of $F$ by constants.
We have shown that $\Gm(\C)$ is a definable quotient of $\Aut_{F}(p/\C)$.
But this is impossible as $\Aut_{F}(p/\C)$ is definably isomorphic to $\operatorname{PSL}_2(\C)$ by Fact~\ref{homspace}.
So $\logd^{-1}(p')$ does not split.

We have established the following.

\begin{theorem}
\label{3dim}
For any differential field $F$ and minimal $\C$-internal type $p\in S_1(F)$ with $\Aut_F(p/\C)$ $3$-dimensional, there exists a minimal $1$-type $p'$, in the definable closure of a nonforking extension of $p$, such that $\logd^{-1}(p')$ is $\C$-internal but does not split.
Moreover, such $p$ and $F$ exist.
\end{theorem}

\bigskip
\section{Dimension not $3$}
\label{sect-not3dim}

\noindent
We know by Fact~\ref{homspace} that a minimal $\C$-internal type has binding group of dimension at most $3$.
We have already considered the case of dimension $3$.
What remains is to consider the remaining three possibilities.

\medskip
\subsection{The dimension $2$ case}
\label{sect-2dim}

\noindent
Fix a differential field $F$.
In this section we consider minimal $\C$-internal types $p\in S_1(F)$ with $2$-dimensional binding groups.
We will show that in that case, if $\logd^{-1}(p)$ is almost $\C$-internal then it splits.

We begin with a general lemma about dependence over logarithmic derivatives.

\begin{lemma}
\label{gu1u2}
Suppose $(u_1,\dots,u_n)$ is a sequence of nonzero elements, and $a_i:=\logd u_i$ for $i=1,\dots,n$.
If the elements $u_1,\dots,u_n$ are field-theoretically algebraically dependent over $\acl(F,\C,a_1,\dots,a_n)$ then they are multiplicatively dependent.
That is, for some $e \in\acl(F,\C,a_1,\dots,a_n)$ and integers $\rho_1,\dots,\rho_n$ not all zero, $e u_1^{\rho_1}\cdots u_n^{\rho_n}=1$.
\end{lemma}

\begin{proof}
Let $u:=(u_1,\dots,u_n)$ and $a:=(a_1,\dots,a_n)$.
Let $P(x)$ be a nonzero polynomial over $\acl(F,\C,a)$ in variables $x=(x_1,\dots,x_n)$ such that $P(u)=0$, and which has the minimum number of terms among all nonzero polynomials over $\acl(F,\C,a)$ with this property.
Write $P(x)$ in multi-index notation as $\displaystyle P(x)=\sum_{\eta\in I}e_\eta x^\eta$ where the $e_\eta$ are nonzero elements of $\acl(F,\C,a)$.
Since
$$\logd(u^\eta)=\eta_1a_1+\cdots+\eta_na_n=:\eta\cdot a,$$
we compute
\begin{eqnarray*}
\delta P(u)
&=&
\sum_{\eta\in I}\big((\delta e_\eta) u^\eta+e_\eta \delta(u^\eta)\big)\\
&=&
\sum_{\eta\in I}\big((\logd e_\eta)e_\eta u^\eta+e_\eta \logd(u^\eta)u^\eta\big)\\
&=&
\sum_{\eta\in I}(\logd e_\eta+\eta\cdot a)e_\eta u^\eta.
\end{eqnarray*}
Now fix $\nu\in I$.
Then, as $P(u)=0$,
\begin{eqnarray*}
0
&=&
(\logd e_\nu+\nu\cdot a)P(u)-\delta P(u)\\
&=&
\sum_{\eta\in I}\big((\logd e_\nu+\nu\cdot a)-(\logd e_\eta+\eta\cdot a)\big)e_\eta u^\eta\\
\end{eqnarray*}
and $\displaystyle Q(x):=\sum_{\eta\in I}\big((\logd e_\nu+\nu\cdot a)-(\logd e_\eta+\eta\cdot a)\big)e_\eta x^\eta$ is a polynomial over $\acl(F,\C,a)$ with at least one less term than $P(x)$.
By minimality, $Q(x)=0$.
That is, $(\logd e_\nu+\nu\cdot a)=(\logd e_\eta+\eta\cdot a)$ for all $\eta\in I$.
In other words, $\logd(e_\nu u^\nu)=\logd(e_\eta u^\eta)$ for all $\eta,\nu\in I$.
Note that $|I|>1$ since each $u_i$ is nonzero as are the $e_\eta$ for $\eta\in I$.
Fixing distinct $\eta,\nu\in I$, let $c\in\C$ be such that $e_\nu u^\nu=ce_\eta u^\eta$.
Letting $\rho:=\nu-\eta$ and $e:=\frac{ce_\eta}{e_\nu}$ we have that $e u^\rho=1$, as desired.
\end{proof}

\begin{proposition}
\label{w3}
Suppose $p\in S_1(F)$ is a minimal type such that $p^{(3)}$ is not weakly $\C$-orthogonal and $\logd^{-1}(p)$ is almost $\C$-internal.
Let $(a_1,a_2)\models p^{(2)}$.
There exist $w\in F\langle a_1,a_2\rangle$ and a nonzero integer $k$ such that $\logd w=k(a_2-a_1)$.
\end{proposition}

\begin{proof}
Extend $(a_1,a_2)$ to a Morley sequence $(a_i:i\geq 1)$ in $p$ and consider a Morley sequence $(u_i:i\geq 1)$ in $\logd^{-1}(p)$ such that $\logd u_i=a_i$ for all $i$.

Since $\logd^{-1}(p)$ is almost $\C$-internal, by Lemma~\ref{aim} there is some $n>0$ such that $u_1\in\acl(F,\C,u_2,\dots,u_n)$.
Since each $\delta u_i=a_iu_i$, this means that $u_1$ is in the field-theoretic algebraic closure of $\{u_2,\dots,u_n\}$ over $\acl(F,\C,a_1,\dots,a_n)$.
Lemma~\ref{gu1u2} therefore applies and we have $e \in\acl(F,\C,a_1,\dots,a_n)$ and integers $\rho_1,\dots,\rho_n$ that are not all zero, such that $e u_1^{\rho_1}\cdots u_n^{\rho_n}=1$.

Without loss of generality, we may assume that $\rho_1\neq 0$ and set $k:=\rho_1$.
Let $\sigma,\tau \in\Aut_F(\U)$ be such that
$$\sigma(u_1,u_2,\dots,u_n)=(u_1, u_3,u_4,\dots,u_{n+1})$$
and
$$\tau(u_1,u_2,\dots,u_n)=(u_2, u_3,u_4,\dots,u_{n+1}).$$
Applying $\sigma$ and $\tau$ to the identity $e u_1^{\rho_1}\cdots u_n^{\rho_n}=1$, we get that
$$\sigma(e) u_1^{k}\prod_{i=2}^n u_{i+1}^{\rho_i}=1$$ and 
$$\tau(e) u_2^{k}\prod_{i=2}^n u_{i+1}^{\rho_i}=1.$$
Hence $\alpha u_1^k=u_2^k$ where $\alpha:=\frac{\sigma(e)}{\tau(e)}$.
But we also have
$$\sigma(a_1,a_2,\dots,a_n)=(a_1,a_3,a_4,\dots,a_{n+1})$$
and
$$\tau(a_1,a_2,\dots,a_n)=(a_2, a_3,a_4,\dots,a_{n+1})$$
so that
$\alpha \in \acl(F,\C,a_1,\dots,a_{n+1})$.
We are given that $p^{(3)}$ is not weakly $\C$-orthogonal, so by minimality $a_3\in\acl(F,\C,a_1,a_2)$.
By indiscernibility, the same holds for all $a_m$ with $m\geq 3$.
Hence $\alpha\in\acl(F,\C,a_1,a_2)$.
And we have $\logd \alpha=k(a_2-a_1)$.
It remains only to push $\alpha$ down into $F\langle a_1,a_2\rangle$

Let $\beta$ be the product of all the distinct $(F,\C,a_1,a_2)$-conjugates of $\alpha$.
Then $\beta\in\dcl(F,\C,a_1,a_2)$ and $\logd(\beta)=nk(a_2-a_1)$ where $n\geq 1$ is the number of distinct conjugates.
Write $\beta=h(c)$ where $h(z)\in F\langle a_1,a_2\rangle(z)$ and $c$ is a tuple of constants.
Then $c$ is a solution to $\logd h(z)=nk(a_2-a_1)$, and so by Lemma~\ref{inc} there must be a solution with co-ordinates in $F\langle a_1,a_2\rangle^{\alg}\cap\C$, say~$c'$.
So $h(c')\in F\langle a_1,a_2\rangle^{\alg}$ and $\logd(h(c'))=nk(a_2-a_1)$.
Now, letting $w$ be the product of the $F\langle a_1,a_2\rangle$-conjugates  of $h(c')$, then $w\in F\langle a_1,a_2\rangle$ and $\logd(w)=mnk(a_2-a_1)$ where $m\geq 1$ is the number of distinct  $F\langle a_1,a_2\rangle$-conjugates  of $h(c')$.
That is, $w$ along with the nonzero integer $mnk$ witnesses the truth of the proposition.
\end{proof}

\begin{theorem}
\label{thm-dim2}
Suppose $F$ is an algebraically closed differential field and $p\in S_1(F)$ is a minimal $\C$-internal type whose binding group is $2$-dimensional.
If $\logd^{-1}(p)$ is almost $\C$-internal then it splits.
\end{theorem}

\begin{proof}
Our strategy is to show that for some nonzero integer $k$ and some $\delta$-rational function $h$ over $F$, if $a\models p$ then $k a+\logd h(a)\in F$.
From this a splitting would follow: Let $u\models \logd^{-1}(p)$ with $a=\logd (u)$.
Set $\displaystyle w_1:=\frac{1}{h(a)}$ and $\displaystyle w_2:=u^{k}h(a)$.
Then $u^{k}=w_1w_2$, $w_1\in F\langle a\rangle$, and $\logd(w_2)=k a+\logd h(a)\in F$, witnessing the splitting of $\logd^{-1}(p)$.

Since the binding group $G:=\Aut_F(p/\C)$ is $2$-dimensional, it follows from Fact~\ref{homspace} that there is a definable field structure, say $\K$, on $p(\U)$, induced by a definable bijection with $\C$, and that with respect to this field structure the action of $G$ is that of $\Gm(\K)\ltimes\Ga(\K)$ acting by affine linear transformations.
Note that the action is $2$-transitive.
Hence, in particular, if $a_1\neq a_2$ realise $p$ then $(a_1,a_2)\models p^{(2)}$.

Fix $(a_1,a_2)\models p^{(2)}$.
By Proposition~\ref{w3} there exist a nonzero integer $k$ and $\delta$-rational function $g(x,y)$ over $F$ such that $\logd g(a_1,a_2)=k(a_2-a_1)$.
Replacing $g(x,y)$ by $\displaystyle\frac{g(x,y)}{g(y,x)}$, and $k$ by $2k$, we may assume that $g(a_1,a_2)=g(a_2,a_1)^{-1}$.

Consider now the element $\sigma\in G$ which is the affine linear transformation that translates by ``$a_2-a_1$" -- here the translation and subtraction is in the sense of the field $\K$.
Note that $\sigma(a_1)=a_2$.
Extend $(a_1,a_2)$ to an infinite sequence by setting $a_{\ell+1}:=\sigma^\ell a_1$ for all $\ell\geq 1$.
Note that as the elements in this sequence are all distinct (as $a_1\neq a_2$), we have that $(a_i,a_j)\models p^{(2)}$ for all $i\neq j$.
For notational convenience, let us write $g_{i,j}:=g(a_i,a_j)$.
In particular,  for all $\ell>2$, $g_{1,2}g_{2,\ell}g_{\ell,1}$ is a constant since it has logarithmic derivative zero.

We claim that in fact
$g_{1,2}g_{2,\ell}g_{\ell,1}=\pm 1$.
Observe that $g_{1,2}g_{2,3}\cdots g_{\ell-1,\ell}g_{\ell,1}\in\C$ for all $\ell>2$ as well.
Hence it is fixed by $\sigma$ and we get
\begin {eqnarray*}
g_{1,2}g_{2,3}\cdots g_{\ell-1,\ell}g_{\ell,1}
&=&
\sigma(g_{1,2}g_{2,3}\cdots g_{\ell-1,\ell}g_{\ell,1})\\
&=&
g_{2,3}g_{3,4}\cdots g_{\ell,\ell+1}g_{\ell+1,2}
\end{eqnarray*}
and so, dividing by $g_{2,3}\cdots g_{\ell-1,\ell}$, we get $g_{1,2}g_{\ell,1}=g_{\ell,\ell+1}g_{\ell+1,2}$,
and so,
\begin{equation}
\label{applysigma}
g_{1,2}g_{2,\ell}g_{\ell,1}=g_{2,\ell}g_{\ell,\ell+1}g_{\ell+1,2}.
\end{equation}
On the other hand, let $\tau\in G$ be the affine linear transformation
$$``x\mapsto -x+\ell a_2-(\ell-2)a_1"$$
where again the linear operations here are in the sense of the field $\K$.
Then we see that $\tau$ swaps $a_1$ with $a_{\ell+1}$, and swaps $a_2$ with $a_\ell$.
So applying $\tau$ to the constant $g_{1,2}g_{2,\ell}g_{\ell,1}$ we get
$g_{1,2}g_{2,\ell}g_{\ell,1} = g_{\ell+1,\ell}g_{\ell,2}g_{2,\ell+1}$, and recalling that $g_{i,j}=g_{j,i}^{-1}$ we deduce that
 \begin{equation}
 \label{applytau}
g_{1,2}g_{2,\ell}g_{\ell,1}=(g_{2,\ell}g_{\ell,\ell+1}g_{\ell+1,2})^{-1}
\end{equation}
Equations~(\ref{applysigma}) and~(\ref{applytau}) together imply that $g_{1,2}g_{2,\ell}g_{\ell,1}=g_{1,2}g_{2,\ell}g_{\ell,1}^{-1}$, so that
$$g_{1,2}g_{2,\ell}g_{\ell,1}=\pm 1$$
for all $\ell>2$, as desired.

We may assume that $g_{1,2}g_{2,\ell}g_{\ell,1}=1$ for infinitely many $\ell>2$, as the other case can be handled similarly.
That is, infinitely many realisations of $p$ satisfy
$$g(a_1,a_2)g(a_2,x)g(x,a_1)=1.$$
By minimality, the generic realisation over $F\langle a_1,a_2\rangle$ satisfies it too.
That is, renaming $a_3$, if we take $(a_1,a_2,a_3)\models p^{(3)}$ then $g(a_1,a_2)g(a_2,a_3)g(a_3,a_1)=1$.

Since $g(a_3,a_1)^{-1}=g(a_1,a_3)$, we get that  $\displaystyle g(a_1,a_2)=\frac{g(a_1,a_3)}{g(a_2,a_3)}$.
Now, $p$ being minimal and almost $\C$-internal, it is of order one, and hence, for $i=1,2$, we can write $g(a_i,a_3)=H_i(a_3,\delta a_3)$ where $H_i(y,z)$ is a rational function over $F\langle a_i\rangle$.
That is, $(a_3,\delta a_3)$ is a solution to the algebraic equation
\begin{equation}
\label{algeq}
g(a_1,a_2)=\frac{H_1(y,z)}{H_2(y,z)}
\end{equation}
over $F\langle a_1,a_2\rangle$.
The algebraic locus of $(a_3,\delta a_3)$ over $F$ is a plane curve $C$, and it agrees with the algebraic locus over $F\langle a_1,a_2\rangle$ (as $a_3$ is generic over $F\langle a_1,a_2\rangle$).
Hence all but finitely many points on $C$ are solutions to~(\ref{algeq}).
In particular, as $F$ is algebraically closed and $C$ is defined over $F$, there is an $F$-point of $C$, say $(u,v)$, satisfying~(\ref{algeq}).
Writing $H_i(u,v)=h_i(a_i)$ for some $\delta$-rational function $h_i$ over $F$, we have that $\displaystyle g(a_1,a_2)=\frac{h_1(a_1)}{h_2(a_2)}$.
By total indiscernibility we have also $\displaystyle g(a_2,a_1)=\frac{h_1(a_2)}{h_2(a_1)}$.
But then,
\begin{eqnarray*}
g(a_1,a_2)^2
&=&
g(a_1,a_2)g(a_2,a_1)^{-1}\\
&=&
\frac{h_1(a_1)h_2(a_1)}{h_2(a_2)h_1(a_2)}.
\end{eqnarray*}
Letting $h=h_1h_2$ and taking logarithmic derivatives we get
$$2k(a_2-a_1)=\logd h(a_1)-\logd h(a_2),$$
and hence $2ka_2+\logd h(a_2)=2ka_1+\logd h(a_1)$.
As $a_1,a_2$ are independent over $F$, this element must be in $F$.
That is, $2ka+\logd h(a)\in F$ for all $a\models p$.
As explained in the first paragraph of the proof, this implies that $\logd^{-1}(p)$ splits.
\end{proof}

The following example shows that Theorem~\ref{thm-dim2} is not vacuous.

\begin{example}
\label{exdim2}
Let $F=\mathbb Q(t)^{\alg}$ where $\delta t=1$.
Consider $\delta x=tx+1$ and let $p$ be the generic type of this equation over $F$.
Then $p$ is a minimal $\C$-internal type whose binding group is $2$-dimensional.
\end{example}

\begin{proof}
Any equation of the form $\delta x=ax+b$ with $a,b\in F$, being an (inhomogeneous) linear equation of order~$1$, has generic type minimal and $\C$-internal.
Now, if the equation has no solution in $F$, its binding group will be the semidirect product of the binding group of the associated linear homogeneous equation $\delta x=ax$ with $\Ga(\C)$.
Indeed, see~\cite[$\S$5.4]{jinthesis} for a direct computation, but it also follows from differential Galois theory -- see for example the proof of Proposition~2.1 of~\cite{singerberman}.
In our case, the homogeneous equation $\delta x=tx$ has binding group $\Gm(\C)$.
So it suffices to show that $\delta x=tx+1$ has no solutions in $F$, as then
$\Aut_F(p/\C)=\Gm(\C)\ltimes\Ga (\C)$.

Suppose $\gamma\in F$ is a solution.
Let $\epsilon$ be the sum of the conjugates of $\gamma$ over $\mathbb Q^{\alg}(t)$.
Then $\epsilon=f(t)$ for some rational function $f$ over $\mathbb Q^{\alg}$,
and
$f'(t)=tf(t)+\ell$.
Writing $f(t)=\frac{P(t)}{Q(t)}$ with $P,Q\in \mathbb Q^{\alg}[t]$ coprime,
a straightforward degree computation shows that $Q$ must be nonconstant.
On the other hand, computing $\frac{f'(t)}{f(t)}$ yields
$$\frac{P'(t)}{P(t)}-\frac{Q'(t)}{Q(t)}=t+\ell\frac{Q(t)}{P(t)}.$$
But this is impossible as every root of $Q$ is a (simple) pole of the left-hand side but not a pole of the right.
\end{proof}

\medskip
\subsection{The dimension $1$ case}
\label{sect-1dim}

\noindent
In this section we focus on when the binding group is of dimension $1$.
Under the additional assumption that it is not an elliptic curve, we are able to show that $\logd^{-1}(p)$ being almost $\C$-internal does imply that it splits.
This is Theorem~\ref{thm-dim1} below.
Actually, we work in the slightly more general setting where we drop the assumption that $p$ is $\C$-internal and instead impose a condition, labelled~$(*)$ in~\ref{thm-dim1}, which amounts to saying that $p$ is interalgebraic with a $\C$-internal type whose binding group is either $\Ga(\C)$ or $\Gm(\C)$.
This increased generality will be used in our intended application, namely for the proof of Theorem~B in Section~\ref{sect-main}.

We begin with the following elementary fact about polynomial algebra.

\begin{lemma}
\label{polyadd}
Suppose $F$ is a field and $g(x,y)\in F(x,y)$ is nonzero.
Let $\hat g(x,y):=g(x,x+y)$ and write $\displaystyle \hat g(x,y)=\frac{\sum_{i=0}^m\alpha_i(y)x^i}{\sum_{i=0}^n\beta_i(y)x^i}$ with $\alpha_0,\dots, \alpha_m, \beta_0,\dots,\beta_n\in F(y)$ and both $\alpha_m(y)$ and $\beta_n(y)$ nonzero.
If $x_1,x_2,x_3$ are indeterminates and $u_1:=x_2-x_1, u_2:=x_3-x_2, u_3:=x_1-x_3$, then
$$g(x_1,x_2)g(x_2,x_3)g(x_3,x_1)=\frac{\alpha_m(u_1)\alpha_m(u_2)\alpha_m(u_3)\big(x_1^{3m}+P(x_1)\big)}{\beta_n(u_1)\beta_n(u_2)\beta_n(u_3)\big(x_1^{3n}+Q(x_1)\big)}$$
where $P(x)$ and $Q(x)$ are polynomials over $F(u_1,u_2,u_3)$ of degree strictly less than $3m$ and $3n$ respectively.
\end{lemma}

\begin{proof}
We first claim that 
$\displaystyle g(x_1,x_2)=\frac{\alpha_m(u_1)\big(x_1^{m}+P_1(x_1)\big)}{\beta_n(u_1)\big(x_1^{n}+Q_1(x_1)\big)}$
where $P_1(x)$ and $Q_1(x)$ are polynomials over $F(u_1)$ of degree strictly less than $m$ and $n$ respectively.
Indeed,
\begin{eqnarray*}
g(x_1,x_2)
&=&
g(x_1,x_1+u_1)\\
&=&
\hat g(x_1,u_1)\\
&=&
\frac{\alpha_m(u_1)\big(x_1^m+\sum_{i=0}^{m-1}\frac{\alpha_i(u_1)}{\alpha_m(u_1)}x_1^i\big)}{\beta_n(u_1)\big(x_1^n+\sum_{i=0}^{n-1}\frac{\beta_i(u_1)}{\beta_n(u_1)}x_1^i\big)}.
\end{eqnarray*}

Next we claim that 
$\displaystyle g(x_2,x_3)=\frac{\alpha_m(u_2)\big(x_1^{m}+P_2(x_1)\big)}{\beta_n(u_2)\big(x_1^{n}+Q_2(x_1)\big)}$
where $P_2(x)$ and $Q_2(x)$ are polynomials over $F(u_1,u_2)$ of degree strictly less than $m$ and $n$ respectively.
Indeed,
\begin{eqnarray*}
g(x_2,x_3)
&=&
g(x_1+u_1,x_1+u_1+u_2)\\
&=&
\hat g(x_1+u_1,u_2)\\
&=&
\frac{\alpha_m(u_2)\big((x_1+u_1)^m+\sum_{i=0}^{m-1}\frac{\alpha_i(u_2)}{\alpha_m(u_2)}(x_1+u_1)^i\big)}{\beta_n(u_2)\big((x_1+u_1)^n+\sum_{i=0}^{n-1}\frac{\beta_i(u_2)}{\beta_n(u_2)}(x_1+u_1)^i\big)}.
\end{eqnarray*}
The second claim follows by noting that $(x_1+u_1)^i$ is of the form $x_1^i+\alpha(x_1)$ where $\alpha(x)$ is a polynomial over $F(u_1)$ of degree $<i$.

Thirdly, we claim that
$\displaystyle g(x_3,x_1)=\frac{\alpha_m(u_3)\big(x_1^{m}+P_3(x_1)\big)}{\beta_n(u_3)\big(x_1^{n}+Q_3(x_1)\big)}$
where $P_3(x)$ and $Q_3(x)$ are polynomials over $F(u_1,u_2,u_3)$ of degree strictly less than $m$ and $n$ respectively.
Indeed,
\begin{eqnarray*}
g(x_3,x_1)
&=&
g(x_1-u_3,x_1)\\
&=&
\hat g(x_1-u_3,u_3)\\
&=&
g_m(u_3)\big((x_1-u_3)^m+\sum_{i=0}^{m-1}\frac{g_i(u_3)}{g_m(u_3)}(x_1-u_3)^i\big)
\end{eqnarray*}
and we finish as in the second claim.

Multiplying the three proven identities together yields the Lemma.
\end{proof}

The following is the multiplicative analogue of the above lemma.

\begin{lemma}
\label{polymult}
Suppose $F$ is a field and $g(x,y)\in F(x,y)$ is nonzero.
Let $\hat g(x,y):=g(x,xy)$ and write $\displaystyle \hat g(x,y)=\frac{\sum_{i=0}^m\alpha_i(y)x^i}{\sum_{i=0}^n\beta_i(y)x^i}$ with $\alpha_0,\dots, \alpha_m, \beta_0,\dots,\beta_n\in F(y)$ and both $\alpha_m(y)$ and $\beta_n(y)$ nonzero.
If $x_1,x_2,x_3$ are indeterminates and $u_1:=\frac{x_2}{x_1}, u_2:=\frac{x_3}{x_2}, u_3:=\frac{x_1}{x_3}$, then
$$g(x_1,x_2)g(x_2,x_3)g(x_3,x_1)=\frac{\alpha_m(u_1)\alpha_m(u_2)\alpha_m(u_3)\frac{u_1^m}{u_3^m}\big(x_1^{3m}+P(x_1)\big)}{\beta_n(u_1)\beta_n(u_2)\beta_n(u_3)\frac{u_1^n}{u_3^n}\big(x_1^{3n}+Q(x_1)\big)}$$
where $P(x)$ and $Q(x)$ are polynomials over $F(u_1,u_2,u_3)$ of degree strictly less than $3m$ and $3n$ respectively.
\end{lemma}

\begin{proof}
The analogous claims are:
\begin{eqnarray*}
g(x_1,x_2)
&=&
\frac{\alpha_m(u_1)\big(x_1^{m}+P_1(x_1)\big)}{\beta_n(u_1)\big(x_1^{n}+Q_1(x_1)\big)}\\
g(x_2,x_3)
&=&
\frac{\alpha_m(u_2)\big(x_1^{m}u_1^m+P_2(x_1)\big)}{\beta_n(u_2)\big(x_1^{n}u_1^n+Q_2(x_1)\big)}\\
g(x_3,x_1)
&=&
\frac{\alpha_m(u_3)\big(\frac{x_1^{m}}{u_3^m}+P_3(x_1)\big)}{\beta_n(u_3)\big(\frac{x_1^{n}}{u_3^n}+Q_3(x_1)\big)}
\end{eqnarray*}
We leave the proofs, which are similar to those of Lemma~\ref{polyadd}, to the reader.
\end{proof}

Returning to our differential-algebraic context, we obtain the following main technical result of this section.

\begin{proposition}
\label{aorm}
Let $F$ be an algebraically closed differential field and $\gamma\in F$ such that the formula $\delta x=\gamma$ has no realisations in $F$.
Suppose there exist rational functions $g\in F(x,y)$ and $f\in F(x)$, and a nonzero integer $k$, such that for $b_1,b_2$ independent realisation of $\delta x=\gamma$ over $F$,
$$\logd g(b_1,b_2)=k\big(f(b_2)-f(b_1)\big).$$
Then there exist a rational function $h\in F(y)$ and a nonzero integer $\ell$, such that
$$\ell kf(b)-\logd h(b)\in F$$
whenever $\delta b=\gamma$.

The same result holds for the formula $\logd x=\gamma$ in place of $\delta x=\gamma$.
\end{proposition}

\begin{proof}
Suppose $q$ is the generic type of $\delta x=\gamma$.
Note that $q$ is minimal, $\C$-internal, and weakly $\C$-orthogonal.
Let $(b_1,b_2,b_3)\models q^{(3)}$.
Our assumption on $g$ implies that $g(b_1,b_2)g(b_2,b_3)g(b_3,b_2)$ has logarithmic derivative zero and hence is a constant.
Our first goal is to modifying $g$ so that $g(b_1,b_2)g(b_2,b_3)g(b_3,b_1)=1$.

Note that $b_1,b_2,b_3$ are algebraic indeterminates over $F$, as are $u_1:=b_2-b_1, u_2:=b_3-b_2, u_3:=b_1-b_3$.
Applying Lemma~\ref{polyadd} we get that
$$g(b_1,b_2)g(b_2,b_3)g(b_3,b_1)=\frac{\alpha_m(u_1)\alpha_m(u_2)\alpha_m(u_3)\big(b_1^{3m}+P(b_1)\big)}{\beta_n(u_1)\beta_n(u_2)\beta_n(u_3)\big(b_1^{3n}+Q(b_1)\big)}$$
where $P(x)$ and $Q(x)$ are polynomials over $F(u_1,u_2,u_3)$ of degree strictly less than $3m$ and $3n$, respectively.
But note that $u_1,u_2,u_3\in \C$ as they are differences of elements having the same derivative.
It follows from the above identity that $\displaystyle \frac{\big(b_1^{3m}+P(b_1)\big)}{\big(b_1^{3n}+Q(b_1)\big)}\in\dcl(F,\C)$.
The fact that $q$ is weakly orthogonal to $\C$ means that $b_1\notin\acl(F,\C)$, and so the only possibility is that $m=n$ and $\displaystyle \frac{\big(b_1^{3m}+P(b_1)\big)}{\big(b_1^{3n}+Q(b_1)\big)}=1$.
We therefore conclude that
$$g(b_1,b_2)g(b_2,b_3)g(b_3,b_1)=\frac{\alpha_m(u_1)\alpha_m(u_2)\alpha_m(u_3)}{\beta_m(u_1)\beta_m(u_2)\beta_m(u_3)}$$
Since $\frac{\alpha_m}{\beta_m}$ is a rational function over $F$ we can write $\frac{\alpha_m(y)}{\beta_m(y)}=\theta(e,y)$ where $\theta$ is a rational function over $F\cap\C$ and $e$ is a tuple from $F$.
We therefore have
\begin{equation}
\label{sepu}
g(b_1,b_2)g(b_2,b_3)g(b_3,b_1)=\theta(e,u_1)\theta(e,u_2)\theta(e,u_3).
\end{equation}

On the other hand, we claim that
\begin{equation}
\label{inb}
g(b_1,b_2)g(b_2,b_3)g(b_3,b_1)\in F^\delta(b_1,b_2,b_3)^{\alg}.
\end{equation}
Indeed, note that $g(b_1,b_2)g(b_2,b_3)g(b_3,b_1)\in F(b_1,b_2,b_3)\cap\C$.
Now, $F(b_1,b_2,b_3)\cap\C$ is a transcendence degree at most two extension of $F(b_1)\cap\C$ since $F(b_1,b_2,b_3)/F(b_1)$ is of transcendence degree two.
But $u_1$ and $u_3$ are algebraically independent over $F(b_1)$ and contained in $F(b_1,b_2,b_3)\cap\C$.
Hence
$$g(b_1,b_2)g(b_2,b_3)g(b_3,b_1)\in\acl(F(b_1)\cap\C, u_1,u_3)\subseteq \acl(F(b_1)\cap\C, b_1,b_2,b_3).$$
Finally, as $q$ is weakly orthogonal to $\C$, $F(b_1)\cap\C=F\cap\C$.

From~(\ref{sepu}) and~(\ref{inb}) we get that the formula $\phi(z)$ given by
$$g(b_1,b_2)g(b_2,b_3)g(b_3,b_1)=\theta(z,u_1)\theta(z,u_2)\theta(z,u_3)$$
is in the $L_{\operatorname{ring}}$-type of $e$ over $F^\delta(b_1,b_2,b_3)^{\alg}$ where $F^\delta:= F\cap\C$.
As $(b_1,b_2,b_3)$ is of transcendence degree $3$ over $F$, we have that $F^\delta(b_1,b_2,b_3)^{\alg}$ is linearly disjoint from $F$ over the algebraically closed field $F^\delta$.
Since $e$ is a tuple from $F$, the formula $\phi(z)$ must have realisation in $F^\delta$, say $\tilde e$.
So each $\theta(\tilde e,u_i)\in\C$ and
$$\frac{g(b_1,b_2)}{\theta(\tilde e,u_1)}\frac{g(b_2,b_3)}{\theta(\tilde e,u_2)}\frac{g(b_3,b_1)}{\theta(\tilde e,u_3)}=1$$
Let $\displaystyle \tilde g(x,y):=\frac{g(x,y)}{\theta(\tilde e,y-x)}$.
Then we have that $\tilde g$ is a rational function over $F$,
$\logd \tilde g(b_1,b_2)=\logd g(b_1,b_2)=k\big(f(b_2)-f(b_1)\big)$,
and
$\tilde g(b_1,b_2)\tilde g(b_2,b_3)\tilde g(b_3,b_1)=1$.
That is, replacing $g$ by $\tilde g$, we can now assume that $g(b_1,b_2)g(b_2,b_3)g(b_3,b_1)=1$.

Next we modify $g$ and $k$ so as to be able to assume that $g(b_1,b_2)^{-1}=g(b_2,b_1)$.
By total indiscernibility of $(b_2,b_3)$ over $F(b_1)$ we have that
$g(b_1,b_3)g(b_3,b_2)g(b_2,b_1)=1$ also.
Hence, letting $\displaystyle \tilde g(x,y):=\frac{g(x,y)}{g(y,x)}$ we retain the fact that $\tilde g(b_1,b_2)\tilde g(b_2,b_3)\tilde g(b_3,b_1)=1$
but now have $\tilde g(b_1,b_2)\tilde g(b_2,b_1)=1$.
Moreover, indiscernibility of $(b_1,b_2)$ over $F$ implies that $\logd\tilde g(b_1,b_2)=2k\big(f(b_2)-f(b_1)\big)$.
So replacing $g$ with $\tilde g$ and $k$ with $2k$, we obtain the desired property that $g(b_1,b_2)^{-1}=g(b_2,b_1)$.

It follows that $\displaystyle g(b_1,b_2)=\frac{g(b_3,b_2)}{g(b_3,b_1)}$.
Since $b_3$ is transcendental over $F(b_1,b_2)$, the equation $\displaystyle g(b_1,b_2)=\frac{g(x,b_2)}{g(x,b_1)}$ has cofinitely many realisations.
In particular it is realised in $F$, say by $e\in F$.
Hence $h(y):=g(e,y)$ is a rational function over $F$ and $\displaystyle g(b_1,b_2)=\frac{h(b_2)}{h(b_1)}$.
It follows that
$$k\big(f(b_2)-f(b_1)\big)=\logd g(b_1,b_2)=\logd h(b_2)-\logd h(b_1)$$
so that 
$kf(b_2)-\logd h(b_2)=kf(b_1)-\logd h(b_1)$.
Since $b_1$ and $b_2$ are independent over $F$ this forces $kf(b)-\logd h(b)\in F$ for all $b\models q$, as desired.

The other case, when $q$ is the generic type of $\logd x=d$, is proved similarly using Lemma~\ref{polymult} rather than~\ref{polyadd}.
We leave the details to the reader.
 \end{proof}

The following theorem was obtained originally, and independently, by the first author in his PhD thesis~\cite[$\S$5]{jinthesis}.

\begin{theorem}
\label{thm-dim1}
Suppose $F$ is an algebraically closed differential field and $p\in S_1(F)$ is a weakly $\C$-orthogonal minimal type satisfying
\begin{itemize}
\item[($*$)]
For $a\models p$ there is $b\in F\langle a\rangle\setminus F$ such that either $\delta b\in F$ or $\logd b\in F$.
\end{itemize}
If $\logd^{-1}(p)$ is almost $\C$-internal then it splits.
\end{theorem}

\begin{remark}
If $\logd^{-1}(p)$ is almost $\C$-internal then so is $p$.
So by Lemma~\ref{iai} $p$ is interalgebraic with a $\C$-internal type $p'$.
Condition~($*$) says that $\Aut_F(p'/\C)$ is either $\Ga(\C)$ or $\Gm(\C)$.
That is, the binding group is of dimension $1$ but is not an elliptic curve.
\end{remark}

\begin{proof}
Fix $b$ is as in~($*$) and let $q:=\tp(b/F)$.
Note that $q$ is weakly $\C$-orthogonal.
On the other hand, as any two realisations of $q$ either differ by a constant or one is a constant multiple of the other, $q^{(2)}$ is not weakly $\C$-orthogonal.
Since $a\in\acl(F,b)$ by minimality, $p^{(2)}$ is also not weakly orthogonal to $\C$.
In particular, $p^{(3)}$ isn't, and Proposition~\ref{w3} applies.
That is, assuming $\logd^{-1}(p)$ is almost $\C$-internal, and fixing $(a_1,a_2)\models p^{(2)}$, we get $w\in F\langle a_1,a_2\rangle$ and nonzero integer $k$ such that $\logd w=k(a_2-a_1)$.

Let $b_1,b_2$ be such that $(a_1b_1,a_2b_2)$ is a Morley sequence in $\tp(ab/F)$.
Note that if in fact $a\in\dcl(F,b)$ then we would have $a_i=f(b_i)$ and $w=g(b_1,b_2)$ for some $f$ and $g$, and we could apply Proposition~\ref{aorm} directly.
But as it is, we have to first deal with the conjugates of $a$.

Write $w=\alpha(a_1,a_2)$ for some $\delta$-rational function $\alpha$ over $F$.
Let $\{a_{i,1},\dots,a_{i,m}\}$ be the set of $(F,b_i)$-conjugates of $a_i$, for $i=1,2$.
In particular, we have a $\delta$-rational function $f(x)$ over $F$ such that $\sum_{r=1}^ma_{i,r}=f(b_i)$.
(Note that $f$ does not depend on $i=1,2$ as by automorphisms the same function will work.)
Also, $\prod_{r=1}^m\prod_{s=1}^m\alpha(a_{1,r},a_{2,s})=g(b_1,b_2)$ for some $\delta$-rational function $g(x,y)$ over~$F$.
And we have that
\begin{eqnarray*}
\logd g(b_1,b_2)
&=&
\logd\left(\prod_{r=1}^m\prod_{s=1}^m\alpha(a_{1,r},a_{2,s})\right)\\
&=&
k\left(\sum_{s=1}^ma_{2,s}-\sum_{r=1}^ma_{1,r}\right)\\
&=&
k\big(f(b_2)-f(b_1)\big).
\end{eqnarray*}
The second equality above uses that for any $r,s\in\{1,\dots,m\}$, the pair $(a_{1,r},a_{2,s})$ realises $p^{(2)}$ and hence $\logd\alpha(a_{1,r},a_{2,s})=k(a_{2,s}-a_{1,r})$.
Indeed, independence follows from the fact that each $a_{i,r}$ is algebraic over $F$ and $b_i$, and $b_1$ is independent of $b_2$ over $F$.
In any case, we see that Proposition~\ref{aorm} applies.
Hence, there exist a rational function $h$ over $F$ and a nonzero integer $\ell$ such that $\logd h(b)=\ell kf(b)-e$ for some $e\in F$.

Consider, now, $\displaystyle \frac{\prod_{s=1}^m\alpha(a_1,a_{2,s})^\ell}{h(b_2)}$.
Note that this is in $F\langle a_1\rangle(c)$ for some constant~$c$.
Indeed, it is clearly in $F\langle a_1,b_2\rangle$, but $b_2$ is an additive or multiplicative translate of $b_1$ by a constant and $b_1\in F\langle a_1\rangle$.
So we can write this element as $\beta(a_1,c)$ where $\beta\in F\langle x\rangle(z)$.
Now compute
\begin{eqnarray*}
\ell k ma_1+\logd\big(\beta(a_1,c)\big)
&=&
\ell k ma_1+\ell\left(\sum_{s=1}^m\logd\alpha(a_1,a_{2,s})\right)-\logd h(b_2)\\
&=&
\ell k ma_1+\ell k\left(\sum_{s=1}^m(a_{2,s}-a_1)\right)-\ell kf(b_2)+e\\
&=&
e\\
&\in&
F.
\end{eqnarray*}
By Lemma~\ref{inc} we can find $\tilde c\in F\langle a_1\rangle\cap \C$ such that $\ell k ma_1+\logd\big(\beta(a_1,\tilde c)\big)=e$.
But by weak $\C$-orthogonality, $F\langle a_1\rangle\cap \C=F^\delta$.
So,  $\theta(x):=\beta(x,\tilde c)$ is over $F$ and we have that $na_1+\logd\theta(a_1)\in F$ for $n:=\ell k m$.

From this, as in the case of Theorem~\ref{thm-dim2}, a splitting follows:
Let $u\models \logd^{-1}(p)$ with $a=\logd (u)$.
Set $\displaystyle w_1:=\frac{1}{\theta(a)}$ and $\displaystyle w_2:=u^{n}\theta(a)$.
Then $u^{n}=w_1w_2$, $w_1\in F\langle a\rangle$, and $\logd(w_2)=na+\logd \theta(a)\in F$.
\end{proof}

\begin{example}
That Theorem~\ref{thm-dim1} is not vacuous is witnessed for example by taking $p$ to be the generic type of $\delta x=x^2$ over $F:=\mathbb Q^{\alg}$.
Then $p$ is minimal and weakly $\C$-orthogonal, and if $a\models p$ then $b:=\frac{1}{a}$ satisfies $\delta b=-1\in F$.
So condition~($*$) holds.
We argued in the Introduction that in this case $\logd^{-1}(p)$ is $\C$-internal (and that it splits).

Another example would be to take the formula $\delta x=x(x-1)$ instead.
If $a\models p$ then this time $b:=\frac{a-1}{a}$ has the property that $\logd b=1\in F$, and so again~($*$) holds but in the other, multiplicative, way.
One can show that in this case too $\logd^{-1}(p)$ is $\C$-internal (and that it splits).
\end{example}

\bigskip
\subsection{The dimension $0$ case}
\label{sect-0dim}

\begin{theorem}
\label{thm-dim0}
Suppose $F$ is an algebraically closed differential field and $p\in S_1(F)$ is a minimal type that is not weakly orthogonal to $\C$.
If $\logd^{-1}(p)$ is almost $\C$-internal then it splits.
\end{theorem}

\begin{proof}
Let $(a_1,a_2,a_3)\models p^{(3)}$.
Since $p^{(3)}$ is not weakly orthogonal to $\C$, if $\logd^{-1}(p)$ is almost $\C$-internal then by Proposition~\ref{w3} there is $w\in F\langle a_1,a_2\rangle$ and a nonzero integer $k$ such that $\logd w=k(a_2-a_1)$.
Write $w=g(a_1,a_2)$ for some $\delta$-rational function $g(x,y)$.
So $\logd g(a_1,a_3)+ka_1=\logd g(a_2,a_3)+ka_2$.
That is, $a_3$ realises the formula $\phi(z)$ over $F\langle a_1,a_2\rangle$ given by
$\logd g(a_1,z)+ka_1=\logd g(a_2,z)+ka_2$.
Since $p$ is of order one (being minimal and nonorthogonal to $\C$) it is the generic type of some strongly minimal set $D$ over $F$.
As $a_3$ is generic in $D$ over $F\langle a_1,a_2\rangle$, we have that cofinitely many points of $D$ realise $\phi(z)$.
But as $p$ is not even weakly $\C$-orthogonal, it cannot be isolated.
Hence $D(F)$ is infinite.
Let $e$ be an $F$-point of $D$ that realises $\phi(z)$.
Then, setting $h(x):=g(x,e)$, we have that $h$ is $\delta$-rational over $F$ and $\logd h(a_1)+ka_1=\logd h(a_2)+ka_2$.
Since $a_1$ and $a_2$ are independent over $F$, we have that $\logd h(a)+ka\in F$ for all $a\models p$.
From this, as we have seen twice before, the splitting of $\logd^{-1}(p)$ follows:
For $u\models \logd^{-1}(p)$ with $a=\logd (u)$, set $\displaystyle w_1:=\frac{1}{h(a)}$ and $\displaystyle w_2:=u^{k}h(a)$.
Then $u^{k}=w_1w_2$, $w_1\in F\langle a\rangle$, and $\logd(w_2)=ka+\logd h(a)\in F$.
\end{proof}

\begin{example}
We give an example of a type satisfying the assumptions of Theorem~\ref{thm-dim0}.
Consider again the formula $\delta x=x^2$ but this time let $p$ be the generic type over $F:=\mathbb Q(t)^{\alg}$ where $\delta t=1$.
Then $p$ is not weakly $\C$-orthogonal, as witnessed by the fact that $\frac{1}{a}+t\in\C$ for $a\models p$.
But it is still the case that $\logd^{-1}(p)$ is $\C$-internal (and splits).
\end{example}

\bigskip
\section{The proofs of Theorems~A and~B}
\label{sect-main}

\noindent
We put things together and deduce our main theorems.

\begin{thmA}
Suppose $F$ is an algebraically closed differential field and $f\in F(x)$ is such that the generic type $p\in S_1(F)$ of $\delta x=f(x)$ is $\C$-internal with binding group not of dimension~$3$.
Then $\logd^{-1}(p)$ is almost $\C$-internal if and only if it splits.
\end{thmA}

\begin{proof}
We have already observed in Section~\ref{sect-splitting}, when introducing splitting, that if $p$ is (almost) $\C$-internal and $\logd^{-1}(p)$ splits then $\logd^{-1}(p)$ is almost $\C$-internal.
For the converse, assume $\logd^{-1}(p)$ is almost $\C$-internal.
Since $p$ is minimal, Fact~\ref{homspace} implies that $\dim\Aut_F(p/\C)\leq 3$.

If $\Aut_F(p/\C)$ is $0$-dimensional then $p$ is not weakly $\C$-orthogonal and Theorem~\ref{thm-dim0} applies, telling us that $\logd^{-1}(p)$ splits.

If $\Aut_F(p/\C)$ is $1$-dimensional then it is $\Ga(\C)$ or $\Gm(\C)$ or $E(\C)$ for some elliptic surve $E$ over the constants.
We first note that the last case is impossible.
Indeed, it would imply that $p(\U)$ is definably isomorphic to $E(\C)$ over some extension $K$ of the base field $F$.
In particular, taking $a$ to realise the nonforking extension of $p$ to~$K$ and taking $e$ to be a generic point of $E(\C)$ over $K$, we would have $K\langle a\rangle=K\langle e\rangle$.
But $K\langle a\rangle=K(a)$ as $\delta a=f(a)$, and this is a genus $0$ function field over~$K$ as $a$ is a transcendental singleton.
Whereas $K\langle e\rangle=K(e)$ as $e$ is a tuple of constants, and $K(e)$ being the rational function field of the elliptic curve is of genus $1$.
So, $\Aut_F(p/\C)$ is either $\Ga(\C)$ or $\Gm(\C)$.
By Lemma~\ref{gmaction}, condition~($*$) of Theorem~\ref{thm-dim1} is satisfied, and that theorem yields the splitting of $\logd^{-1}(p)$.

Finally, if $\Aut_F(p/\C)$ is $2$-dimensional then Theorem~\ref{thm-dim2} applies and tells us that $\logd^{-1}(p)$ splits.
\end{proof}

We specialise to constant parameters and obtain the following application.

\begin{thmB}
Suppose $F$ is an algebraically closed field of constants and $f\in F(x)$.
The rational vector field defined by
$\begin{Bmatrix}
y'=xy\\
x'=f(x)
\end{Bmatrix}$
is almost internal to the constants if and only if the following conditions on $f$ are satisfied:
\begin{itemize}
\item[(i)]
$f\neq 0$, and
\item[(ii)]
$\displaystyle \frac{1}{f}=\frac{d}{dx}(g)$ or $\displaystyle \frac{1}{f}=\frac{c\frac{d}{dx}(g)}{g}$ for some $c\in F$ and $g\in F(x)$, and
\item[(iii)]
$\displaystyle \frac{kx-e}{f}=\frac{\frac{d}{dx}(h)}{h}$ for some nonzero $k\in\mathbb Z$, $e\in F$ and $h\in F(x)$.
\end{itemize}
\end{thmB}

\begin{proof}
Outside of $y=0$ the given vector field corresponds to the second order differential equation $\displaystyle \delta\left(\logd y\right)=f\left(\logd y\right)$.
This is the equation for $\logd^{-1}(D)$ where $D\subseteq\mathbb A^1$ is defined by $\delta x=f(x)$.
The vector field being almost internal to the constants is therefore equivalent to $\logd^{-1}(D)$ being almost $\C$-internal.
Note that if $f=0$ then $D=\C$ and $\logd^{-1}(\C)$ is well known to not be almost $\C$-internal.
Hence, it suffices to assume $f\neq 0$ and show that $\logd^{-1}(D)$ is almost $\C$-internal if and only if~(ii) and~(iii) hold.

Assume~(ii) and~(iii) hold.
Condition~(ii) already ensures that $D$ is almost $\C$-internal, see~\cite[Theorem 2.8]{mcgrail}.
We show that $\logd^{-1}(D)$ is also almost $\C$-internal.

Let $u\in\logd^{-1}(D)$ and set $a:=\logd u$.
If $a\in F$ then $\tp(u/F)$ is $\C$-internal as $\logd$ has cosets of $\Gm(\C)$ as its fibres.
So we may assume $a\notin F$, and in particular the rational functions $f$ and $h$ are defined and nonzero at $a$.

Let $w_1:=h(a)$.
Since $a\in D$, $\tp(a/F)$, and hence $\tp(w_1/F)$, is almost $\C$-internal.

Let $\displaystyle w_2:=\frac{u^k}{h(a)}$.
We compute that
\begin{eqnarray*}
\logd w_2
&=&
ka-\logd h(a)\\
&=&
ka-\left(\frac{\dx(h)(a)\delta(a)}{h(a)}\right)\ \ \ \ \text{ as $h$ is over constant parameters}\\
&=&
ka-\left(\frac{\dx(h)(a)f(a)}{h(a)}\right)\ \ \ \ \text{ as $a\in D$}\\
&=&
ka-\left(\frac{ka-e}{f(a)}f(a)\right) \ \ \ \ \ \ \text{ by~(iii)}\\
&=&
e
\end{eqnarray*}
and so $\logd w_2\in F$.
Hence $\tp(w_2/F)$ is $\C$-internal.

Since $u^k=w_1w_2$, we have that $\tp(u/F)$ is almost $\C$-internal, as desired.

For the converse, assume that $\logd^{-1}(D)$ is almost $\C$-internal.
In particular, $D$ is nonorthogonal to $\C$.
This means that after passing to a differentially closed field $K$ extending $F$, which we may assume has $F$ as its field of constants, and letting $a\in D$ be generic over $K$ (i.e., not in $K$), it will be the case that $K(a)$ has new constants -- that is constants not in $F$.
Now, Rosenlicht's theorem~\cite[Proposition 2]{rosenlicht}, tells us exactly that condition~(ii) holds.

Let $p\in S_1(F)$ be the generic type of $D$.
Using~(ii), we show that $p$ satisfies condition~($*$) of Theorem~\ref{thm-dim1}.
Indeed, suppose $\displaystyle \frac{1}{f}=\frac{c\frac{d}{dx}(g)}{g}$ for some $c\in F$ and $g\in F(x)$.
Then we compute, that for $a\models p$,
\begin{eqnarray*}
\logd g(a)
&=&
\frac{\dx(g)(a)\delta(a)}{g(a)}\ \ \ \ \text{ as $g$ is over constant parameters}\\
&=&
\frac{\dx(g)(a)}{g(a)}f(a)\ \ \ \ \text{ as $a\in D$}\\
&=&
\frac{1}{cf(a)}f(a)\\
&=&
\frac{1}{c}.
\end{eqnarray*}
So $b:=g(a)$ satisfies $\logd b\in F$.
If, on the other hand, (ii) takes the form of $\displaystyle \frac{1}{f}=\frac{d}{dx}(g)$, a similar computation shows that $\delta b\in F$.
Hence~$(*)$ is satisfied.

Applying Theorem~\ref{thm-dim1} or Theorem~\ref{thm-dim0} to $p$, depending on whether $p$ is weakly $\C$-orthogonal or not, we have that for $u\models\logd^{-1}(p)$, some integer power of $u$ factors in $\Gm$ as $u^k=w_1w_2$ where $w_1\in\dcl(F,\logd u)$ and $\logd w_2\in F$.
We have that $k\logd u=\logd w_1+\logd w_2$.
Let $a:=\logd u$, and write $w_1=h(a)$ for some $h\in F(x)$.
(Note that as $a\in D$, $\dcl(F,a)=F(a)$.)
Let $e:=\logd w_2$.
Then,
\begin{eqnarray*}
\frac{ka-e}{f(a)}
&=&
\frac{k\logd u-\logd w_2}{\delta(a)}\\
&=&
\frac{\logd w_1}{\delta(a)}\\
&=&
\frac{\dx(h)(a)\delta(a)}{h(a)\delta(a)}\\
&=&
\frac{\dx(h)(a)}{h(a)}.
\end{eqnarray*}
Since $a$ is a transcendental, we have that
$\displaystyle \frac{kx-e}{f}=\frac{\dx(h)}{h}$ in $F(x)$, as desired.
\end{proof}

\vfill
\pagebreak

%\bibliographystyle{plain}
%\bibliography{../principal}

\end{document}